\documentclass[11pt,reqno]{amsart}
\usepackage{newlfont,amsmath,amsthm,amsgen,amscd}
\usepackage{amssymb}

\usepackage{amsmath}
\usepackage{latexsym}
\usepackage{fancyhdr}
\usepackage{ifthen}

\usepackage{amsfonts}
\usepackage{lastpage}
\usepackage{afterpage}
\usepackage{epsfig}
\usepackage{enumerate}
\usepackage{euscript,color}

\newcommand{\hook}{\makebox[7pt]{\rule{6pt}{.3pt}\rule{.3pt}{5pt}}\,}

\renewcommand{\eqref}[1]{(\ref{#1})}

\newcommand{\belabel}[1]{\begin{equation}\label{#1}}
\newtheorem{mthm}{Theorem}

\newcommand{\C}{{\mathbb C}}

\newcommand{\rr}{\mathbb{R}}
\newcommand{\N}{{\mathbb N}}

\newcommand{\hol}{\mathfrak{hol}}

\newcommand{\bM}{{\overline{\cal M}}}

\newcommand{\bg}{{\overline{\mathbf{g}}}}
\newcommand{\bnab}{{\overline{\nabla}}}
\newcommand{\bR}{{\overline{\mathrm{R}}}}
\newcommand{\R}{{\mathrm{R}}}
\newcommand{\bRic}{{\overline{\mathrm{Ric}}}}
\newcommand{\bscal}{{\overline{\mathrm{scal}}}}

\renewcommand{\div}{{\mathrm{div}}}

\newcommand{\tr}{\mathrm{tr}}

\newcommand{\F}{{\cal F}}
\newcommand{\beq}{\begin{eqnarray*}}
\newcommand{\eeq}{\end{eqnarray*}}
\newcommand{\be}{\begin{eqnarray}}
\newcommand{\ee}{\end{eqnarray}}
\newcommand{\Ric}{{\mathrm{Ric}}}
\newcommand{\beqn}{\begin{equation}}
\newcommand{\eeqn}{\end{equation}}

\newcommand{\Hol}{\mathrm{Hol}}

\newcommand{\s}{\mathrm{s}}
\newcommand{\+}{\oplus}

\newcommand{\G}{\mathbf{G}}
\newcommand{\Spin}{\mathbf{Spin}}

\newcommand{\SO}{\mathbf{SO}}

\newcommand{\Sp}{\mathbf{Sp}}

\theoremstyle{definition}

\newtheorem{re}{Remark}[section]

\newtheorem*{bsp*}{Example}
\newtheorem*{def*}{Definition}
\theoremstyle{plain}
\newtheorem{Lemma}{Lemma}[section]
\newtheorem*{lem*}{Lemma}
\newtheorem{Proposition}{Proposition}[section]
\newtheorem{Corollary}{Corollary}[section]
\newtheorem{Theorem}{Theorem}[section]
\newtheorem*{theo*}{Theorem}
\newtheorem*{conj*}{Conjecture}
\newtheorem*{quest}{Open Questions}

\newcommand{\M}{{\cal M}}

\newcommand{\g}{\mathbf{g}}
\newcommand{\h}{\mathbf{h}}

\newcommand{\W}{\mathrm{W}}

\newcommand{\Z}{\mathrm{Z}}

\newcommand{\scal}{\mathrm{scal}}
\newcommand{\del}{\partial}
\renewcommand{\t}{{(t)}}

\newcommand{\bleml}[1]{\begin{Lemma} \label{#1}}
\newcommand{\blem}{\begin{Lemma}}
\newcommand{\elem}{\end{Lemma}}

\newcommand{\btheo}{\begin{Theorem}}
\newcommand{\btheol}[1]{\begin{Theorem}\label{#1}}
\newcommand{\etheo}{\end{Theorem}}

\newcommand{\bpropl}[1]{\begin{Proposition} \label{#1}}
\newcommand{\bprop}{\begin{Proposition}}
\newcommand{\eprop}{\end{Proposition}}

\newcommand{\bcorl}[1]{\begin{Corollary} \label{#1}}
\newcommand{\bcor}{\begin{Corollary}}
\newcommand{\ecor}{\end{Corollary}}

\newcommand{\bbem}{\begin{re}}
\newcommand{\ebem}{\end{re}}

\newcommand{\bprf}{\begin{proof}}
\newcommand{\eprf}{\end{proof}}
\newcommand{\pr}{\mathrm{pr}}

\numberwithin{equation}{section}

\begin{document}
\title[Hyperbolic evolution equations,  Lorentzian   holonomy and Killing spinors]{Hyperbolic evolution equations, Lorentzian   holonomy, and Riemannian generalised Killing spinors}

\author{Thomas Leistner}
\address[TL]{School of Mathematical Sciences, University of Adelaide, SA~5005,
Australia} \email{thomas.leistner@adelaide.edu.au}

\author{Andree Lischewski} 
\address[AL]{Insitut f{\"u}r Mathematik, Humboldt Universit{\"a}t zu Berlin, Unter den Linden~6, 10117 Berlin, Germany} \email{lischews@math.hu-berlin.de}

\thanks{
This work was supported by the Australian Research Council via  the grants FT110100429 and DP120104582.
}
\begin{abstract}
\begin{sloppypar}
We prove that the Cauchy problem for parallel null vector fields on smooth Lorentzian manifolds is well posed. The proof is based on the derivation and analysis of suitable hyperbolic evolution equations  given in terms of the Ricci tensor and other geometric objects.
Moreover, we   classify  Riemannian manifolds satisfying the constraint conditions for this Cauchy problem. 
 It is then possible to characterise certain  holonomy reductions of globally hyperbolic manifolds with parallel null vector in terms of flow equations for Riemannian special holonomy metrics. For exceptional holonomy groups these flow equations have been investigated in the literature before in other contexts. As an application, the results provide a classification of Riemannian manifolds admitting imaginary generalised Killing spinors. We will also give new local normal forms for Lorentzian metrics with parallel null spinor in any dimension.\end{sloppypar}
\end{abstract}

\maketitle


\section{Background and  main results}

Lorentzian manifolds  with parallel null vector fields or parallel null spinor fields arise naturally in geometric as well as physical contexts. In  general relativity they occur as wave-like solutions to the Einstein equations and in string theory they constitute supergravity backgrounds with a high degree of supersymmetry. In geometry they form a class of Lorentzian manifolds with {\em special holonomy}, i.e., whose holonomy group is reduced but the manifold is not locally a product.  The holonomy algebras associated to Lorentzian manifolds with special holonomy were classified in \cite{bb-ike93,leistnerjdg} and {\em local metrics} realising the given holonomy algebras are constructed in \cite{galaev05}. More recently, the interplay between special Lorentzian holonomy, or more specifically, the existence of parallel null vector fields, and {\em global} geometric properties has become the focus of research \cite{candela-flores-sanchez03, bazaikin09,BaumLarzLeistner14,Schliebner15,leistner-schliebner13,BoubelMounoud16}. In the present paper we address the problem of constructing {\em globally hyperbolic} Lorentzian manifolds  by solving a Cauchy problem that arises from the existence of a parallel null vector field. 
Another motivation arises from spin geometry: the existence of a parallel null spinor implies the existence of a parallel null vector field and  it turns out that the associated  Cauchy problem provides some interesting relations to Riemannian spin geometry. In fact, it naturally leads to a classification   result for (complete) Riemannian manifolds admitting so-called {\em imaginary generalised Killing spinors}.

As in a preceding paper \cite{BaumLeistnerLischewski16} we have shown that the Cauchy problem for parallel null vector fields is well posed for {\em real analytic data} and that it always has a globally hyperbolic solution.
The proof  rested on the derivation of suitable evolution equations which can be analysed using the Cauchy-Kowalevski Theorem. More precisely, let $(\bM,\bg)$ be a Lorentzian manifold admitting a nontrivial vector field $V \in \mathfrak{X}(\bM)$ which is parallel with respect to the Levi-Civita connection $\bnab$ of $\bg$ and satisfies $\bg(V,V) = 0$. Suppose further that $(\M,\g)$ is a spacelike hypersurface of $(\bM,\bg)$ which embeds into $\bM$ with Weingarten tensor $\W$. As seen in \cite{BaumLeistnerLischewski16}, requiring the vector field  $V$ to be parallel, imposes on $(\M,\g)$ the constraint
\begin{align}
\nabla U + u \W = 0, \label{cons}
\end{align} 
where $U \in \mathfrak{X}(\M)$ is the vector field given by the negative of the  projection of $V$ onto $T\M$ and $u = \sqrt{\g(U,U)}$. Then, using the Cauchy-Kowalevski Theorem, in \cite{BaumLeistnerLischewski16} it was shown that {\em real analytic} initial data $(\M,\g,U,\W)$ satisfying the constraint (\ref{cons}) can be extended to a Lorentzian manifold $(\bM,\bg)$ with parallel null vector field $V$. 

This result immediately suggest the question whether  the Cauchy problem for a parallel null vector field is also well-posed for {\em smooth} data. For a parallel null spinor this was verified in  \cite{Lischewski15spinor} using techniques surrounding the Cauchy problem for the vacuum Einstein equations. It turns out that these techniques can also be applied here --- after overcoming some difficulties explained below ---  allowing us to prove our  main theorem  that the Cauchy problem for parallel null vector fields is also well posed for {\em smooth data}:

\begin{mthm} \label{cauchy-vf-theo}
Let $(\M,\g)$ be a smooth Riemannian manifold admitting a nontrivial vector field $U$ solving \eqref{cons} for some symmetric endomorphism $\W$ on $\M$. Moreover, let $\lambda \in C^{\infty}(\M, \mathbb{R}^*)$ be a given function. 
Then there exists an open neighbourhood $\bM$ of $\{0\} \times \M$ in $\mathbb{R} \times \M$ and a Lorentzian metric of the form \[\bg = -\widetilde{\lambda}^2 dt^2 + \g_t,\] where $\g_t$ is a family of Riemannian metrics on $\M$ 
and $\widetilde{\lambda}$ is a positive function on $\bM$ with 
\[
\g_0 = \g, \qquad \widetilde{\lambda}|_{\M} = \lambda,\] such that $U$ extends to a parallel null vector field on $(\bM,\bg)$. Moreover, $(\bM,\bg)$ can be chosen to be globally hyperbolic with spacelike Cauchy hypersurface $\M$, i.e., $\M$ is met by every inextendible timelike curve in $(\bM,\bg)$ exactly once.
\end{mthm}
The proof of Theorem \ref{cauchy-vf-theo} in Section \ref{sec3} is based on the theory of quasilinear symmetric hyperbolic PDEs as known from general relativity.  Let us point out, however, one {\em fundamental difference} to related Cauchy problems in general relativity or in \cite{Lischewski15spinor}:

Considering the Cauchy problem for the Einstein equations in general relativity, it follows by definition of the problem that evolution equations for the metric are given in terms of the Ricci (or Einstein) tensor. Similarly, there are integrability conditions for parallel null spinors on Lorentzian manifolds formulated in terms of the Ricci tensor (the Ricci-tensor is  nilpotent \cite{Baum85}) and they lead to obvious evolution equations for the metric in the Cauchy problem for parallel null spinor fields \cite{Lischewski15spinor}. It is important for a smooth solution theory to have evolution equations in terms of the Ricci tensor in these cases because in Lorentzian signature the resulting PDEs can be reformulated as hyperbolic systems, see \cite{FischerMarsden72-1} for instance. In contrast to that, the existence of a parallel null vector field on a Lorentzian manifold yields hardly any nontrivial information about the Ricci tensor. Thus, it is not obvious at all that the methods that work for the Cauchy problem for the Einstein equations or a parallel null spinor field also work for a parallel null vector field and that Theorem \ref{cauchy-vf-theo} can be proved by deriving an evolution equation for the metric $\bg$ in terms of the Ricci tensor. The key idea here is to simply introduce the Ricci tensor as new unknown $\Z=\Ric(\bg) $, consider this as an evolution equation for the metric $\bg$,  and then {\em close the system} by further differentiation that results in  a first order equation for $\Z$. The resulting PDE turns out to be hyperbolic and is a key ingredient for the proof of Theorem \ref{cauchy-vf-theo}.  We believe that this approach can be used in other settings, for example for the Einstein equations with  complicated  energy momentum tensor.

With the result of Theorem \ref{cauchy-vf-theo} at hand,  it is  natural to search for Riemannian manifolds solving the constraint equations \eqref{cons},  in particular for {\em geodesically complete} solutions. As a first step, we exhibit the local structure of solutions to \eqref{cons}. Using the flow of the vector field $U$  in \eqref{cons}, which defines a {\em closed} one-form,  it easily follows that the metric $\g$ can be brought into a adapted normal form:

\begin{mthm}\label{pf}
Any solution $(\M,\g)$ to \eqref{cons} with nowhere vanishing $U$ is locally isometric to
\begin{align}
\left(  \mathcal{I} \times \F, \g = u^{-2} ds^2 + {\h}_s \right), \label{nform-vf}
\end{align}
where $\mathcal{I} \subset \mathbb{R}$ is an interval, ${\h}_s$ is a family of Riemannian metrics on some manifold $\mathcal{F}$ parametrised by $s\in \mathcal{I}$, and  $u^2 = \g(U,U)$. Under this isometry, $U ={u^2}  \partial_s$ and, writing  the function $u$ as a family $u_s = u(s, \cdot)$ of $s$-dependent functions on $\cal F$, in the decomposition $T\M = \mathbb{R}\partial_s \oplus T\F$  we have
\begin{align}
\W= -\frac{1}{u_s} \g(\nabla U, \cdot) = - \begin{pmatrix} \frac{\dot{u}_s}{u^2_s}&\frac{1}{u_s^2} \mathrm{grad}^{{\h}_s}(u_s) \\ \frac{1}{u_s^2}d(u_s) & \frac{u_s}{2}\dot{\h}_s \end{pmatrix}, \label{www}
\end{align} 
where the dot denotes the Lie derivative $\cal L_{\partial_s}$ in $s$-direction.
Moreover, if the vector field $\frac{1}{u^2} U$ is complete, then the universal cover of $\M$ is globally isometric to a manifold of the form (\ref{nform-vf}) with $\cal I=\rr$  and $\cal F$ simply connected.

Conversely, given $(\M,\g)$ as in (\ref{nform-vf}) with $\mathcal{I}=\mathbb{R}$ or $\mathcal{I}=S^1$ a circle, the vector field $U =u^2  \partial_s$ solves \eqref{cons} for $\W$ as in \eqref{www}.  If in addition $\cal F$ is compact and $u$ bounded, then $(\M,\g)$ is complete.\end{mthm}
This theorem gives a (local) classification of Riemanian manifolds satisfying the constraint. It also gives a  method to construct solutions to the constraint equation, in particular {\em complete solutions}: for compact  $\cal F$, if $\cal I=S^1$, the Riemannian manifold   $(\M,\g)$ is complete by compactness, but 
 in Section \ref{complsec}
we  show that this also holds when  $\cal I=\rr$ and $u$  is bounded.

To any manifold $(\M,\g)$ as in \eqref{nform-vf} we can apply Theorem \ref{cauchy-vf-theo}. Let us from now on assume that $\M$ is oriented. It follows that there is a naturally induced orientation on the manifold $\bM$ arising via Theorem \ref{cauchy-vf-theo}. As there is also a parallel null vector on $(\bM,\bg)$ it follows that the holonomy group $\Hol(\bM,\bg)$  of $(\bM,\bg)$, i.e., the group of parallel transports along closed loops, satisfies  
  \[\Hol(\bM,\bg) \subset \SO(n) \ltimes \mathbb{R}^n \subset \SO(1,n+1),\] 
where $\SO(n) \ltimes \mathbb{R}^n$ is the stabiliser in $\SO(1,n+1)$ of a null vector. 
In this case, the main ingredient of $\Hol(\bM,\bg)$  then is the {\em screen holonomy} 
\[{G}:=\mathrm{pr}_{\SO(n)}\Hol(\bM,\bg) \subset \SO(n).\] In \cite{leistnerjdg} it was shown that (the connected component of) the screen holonomy $G$ is always a {\em Riemannian holonomy group}, and hence a product of the groups on Berger's lists \cite{berger55,berger57}.  It is now natural to ask whether one can prescribe ${G}$ by imposing additional conditions on the initial data, i.e., on the family of metrics ${\h}_s$ on $\mathcal{F} \subset \M$. We show that this is indeed the case when $G$ arises as stabiliser of some tensor:


\begin{mthm} \label{doublecone}
Let $(\M,\g,\W,U)$ be given as in \eqref{nform-vf} and \eqref{www} and let $(\bM,\bg)$ be the Lorentzian manifold arising from this choice of initial data via Theorem \ref{cauchy-vf-theo} (for arbitrary choice for $\lambda$). Then $G=\mathrm{pr}_{\SO(n)}\Hol(\bM,\bg) \subset \SO(n)$ lies in the stabiliser of some tensor in $T^{k,l}\mathbb{R}^n$ if and only if there is an $s$-dependent and $\nabla^{{\h}_s}$-parallel family of tensor fields $\eta_s$  on $\mathcal{F}$, of the same type and  subject to the flow equation
\begin{align}
\dot{\eta}_s =-\tfrac{1}{2} \dot{\h}_s^{\sharp} \bullet \eta_s. \label{floweq}
\end{align}
Here, the dot denotes the Lie derivative of a tensor with respect to $\partial_s$, e.g., 
$\dot{\eta}_s:=\mathcal{L}_{\partial_s} \eta_s$, 
 $\dot{\h}^{\sharp} \bullet$ denotes the natural action of the endomorphism $\dot{\h}^{\sharp} \in \text{End} (T\mathcal{F})$ on tensors in $T^{k,l}\mathcal{F}$, and $\sharp$ indicates  the dualisation with respect to $\h_s$.
Moreover:
\begin{enumerate}
\item 
There are proper subgroups $H_1$ and $H_2$ of $\SO(n)$ such that
$G\subset {H}_1 \times {H}_2 $ if and only if 
there is a  local metric splitting
\begin{align}
(\mathcal{F},{\h}_s) \cong (\mathcal{F}_1 \times \mathcal{F}_2, {\h}_s^1+{\h}_s^2)
\end{align} 
with $\Hol(\mathcal{F}_i,{\h}_s^i) \subset {H}_i$. 
\item If  $G$ is contained in one of $\mathbf{SU}(m)$, $\mathbf{Sp}(k)$, $\mathbf{G}_2$, $\mathbf{Spin}(7)$ or trivial, this translates into the  conditions for Riemannian special holonomy metrics from Table \ref{table1}.
\end{enumerate}
\end{mthm}
\begin{table}[h!]\label{table1}  \centering
  \renewcommand{\arraystretch}{1.6}
  \begin{tabular}{|c|c|c|} 
    \hline
$\mathrm{dim}(\mathcal{F})$ & \mbox{condition on } $\mathcal{F}$ & $ \Hol(\bM,\bg) \subset$ \\
    \hline\hline
$2m$ &  $\begin{array}{c} (\mathcal{F},\omega_s,J_s,{\h}_s=\omega_s(J_s\cdot,\cdot)) \mbox{ Ricci-flat Kaehler},\\
\dot{J}_s =-\tfrac{1}{2} \dot{\h}_s^{\sharp} \bullet J_s,\ \delta^{\h_s}(\dot{\h}_s) = 0 \end{array}$ & $\mathbf{SU}(m) \ltimes \mathbb{R}^{2m}$ \\ \hline
$4k$ &  $\begin{array}{c} (\mathcal{F},\omega^i_s,J^i_s,{\h}_s=\omega^i_s(J^i_s\cdot,\cdot))_{i=1,2,3} \mbox{ hyper-Kaehler,}
\\
\dot{J}^i_s =-\tfrac{1}{2} \dot{\h}_s^{\sharp} \bullet J^i_s \end{array}$ & $\mathbf{Sp}(k) \ltimes \mathbb{R}^{4k}$ \\ \hline
7 &  $\begin{array}{c} (\mathcal{F},\phi_s \in \Omega^3(\F)),{\h}_s={\h}_s(\phi(s)) \mbox { }\mathbf{G}_2 \mbox{ metrics,}\\
\dot{\phi}_s =-\tfrac{1}{2} \dot{\h}_s^{\sharp} \bullet \phi_s  \end{array}$ & 
$\mathbf{G}_2 \ltimes \mathbb{R}^7$ \\ \hline
8 &  $\begin{array}{c} (\mathcal{F},\psi_s \in \Omega^4(\F), {\h}_s={\h}_s(\psi_s)) \mbox { }\mathbf{Spin}(7) \mbox{ metrics,}\\
\dot{\psi}_s =-\tfrac{1}{2} \dot{\h}_s^{\sharp} \bullet \psi_s  \end{array}$ & 
$\mathbf{Spin}(7) \ltimes \mathbb{R}^8$ \\ \hline
$n$ &  $\begin{array}{c}{\h}_s \mbox{ flat} \end{array}$ & 
$\mathbb{R}^n$ \\ \hline
  \end{tabular}
  \vspace{8pt}
  \caption{Equivalent characterisation of special screen holonomy for $(\bM,\bg)$ in terms of flow equations for tensors on $\mathcal{F}$. }
\end{table}
In Table \ref{table1} we write ${\h}_s = {\h}_s(\phi_s)$ and ${\h}_s= {\h}_s(\psi_s)$ to indicate  that for families of $\mathbf{G}_2$ and $\mathbf{Spin}(7)$ structures, the metric $\h_s$ is defined algebraically in terms of a distinguished stable $3$-form $\phi_s$ or a generic  $4$-form $\psi_s$, respectively. The explicit formulas can be found for example in  \cite{bryant87,joyce07,Karigiannis05}.
In particular, Theorems \ref{cauchy-vf-theo} and \ref{doublecone} provide a construction principle for Lorentzian manifolds with reduced screen holonomy. Obviously, warped products $(\cal I \times \mathcal{F}, \g = ds^2 + f(s)  \h_0)$ with $(\mathcal{F},\h_0)$ being a Ricci-flat special holonomy manifold, i.e., $\Hol(\mathcal{F},h_0) \in \{\mathbf{SU}(m),\mathbf{Sp}(k),\mathbf{G}_2,\mathbf{Spin}(7) \}$ or trivial, are obvious examples for the construction in Theorem \ref{doublecone}.

In the final part of this article we turn to applications of these results and constructions.
As a first application of Theorems \ref{cauchy-vf-theo} and \ref{doublecone}, we address a classification problem in \textit{Riemannian} spin geometry. In doing so, we have to change our point of view slightly: so far, the object of interest was the Lorentzian manifold $(\bM,\bg)$ constructed from initial data $(\M,\g,\W)$ via Theorems \ref{cauchy-vf-theo} or \ref{doublecone}. In Section \ref{igr}, however, $(\bM,\bg)$ is regarded as an auxiliary object and we show how the detailed study of $\Hol(\bM,\bg)$ from the previous statements can in turn be used to prove  a partial classification of Riemannian manifolds $(\M,\g)$ admitting \textit{ imaginary }$\W$\textit{-Killing spinors}. By definition, these are sections $\varphi $ of the complex spinor bundle $\mathbb{S}^{\g}\to \M$ of $(\M,\g)$  satisfying
\begin{align}
\nabla^{\mathbb{S}^{\g}}_X \varphi = \frac{\mathrm{i}}{2} \  \W(X) \cdot \varphi, \label{genks}
\end{align}
for some $\g$-symmetric endomorphism $\W$. Here, $\cdot$ denotes Clifford multiplication. Clearly, condition \eqref{genks} arises as a generalisation of the equation for imaginary Killing spinors, for which  $\W = \frac{\mathrm{ i}}{2} \  \mathrm{Id}$, see  \cite{bfgk91}. Moreover, solutions to equation (\ref{genks}) are the counterpart to real generalised Killing spinors which have been in the focus of recent research, for example in \cite{AmmannMoroianuMoroianu13,baer-gauduchon-moroianu05}. Given a solution to \eqref{genks}, we denote by $U_{\varphi} \in \mathfrak{X}(\M)$ the Dirac current of $\varphi$, given by 
\begin{equation}
\label{diracU}
\g(U_{\varphi},X) = - i \,(X \cdot \varphi,\varphi),\quad \text{  for all $X\in T\M$,}\end{equation}
 and assume that $\varphi$ solves the algebraic constraint
\begin{align}
 U_{\varphi} \cdot \varphi =\mathrm{ i}\,u_{\varphi} \,\varphi, \label{spinor-2}
\end{align}
where $u_{\varphi}= \sqrt{\g(U_{\varphi},U_{\varphi})} = \|\varphi\|^2$. This constraint is known to hold for imaginary Killing spinors, i.e. $\W = \lambda \ \mathrm{ Id}$, and it can also be motivated from the perspective of Lorentzian manifolds with parallel null spinors, cf. \cite{BaumLeistnerLischewski16}. We obtain the following classification result which generalises results from \cite{baum89-3,baum89-2}, see also \cite{bfgk91}, where it is shown that in the complete case and for $ \W = f \ \mathrm{ Id}$, $(\M,\g)$ is necessarily isometric to a warped product.

\begin{mthm} \label{imks}
Let $(\M,\g)$ be a Riemannian spin manifold admitting an imaginary $\W$-Killing spinor $\varphi$ satisfying the algebraic equation \eqref{spinor-2}. Then:
\begin{enumerate}
\item \label{imks1}$(\M,\g)$ is locally isometric to
 \begin{align}
 (\M,\g) = \left(\cal I \times \mathcal{F}_1 \times...\times \mathcal{F}_k, \g = \frac{1}{u^2} ds^2 + \h^1_s + ...+\h^k_s \right)\label{dataf}
 \end{align}
for Riemannian manifolds $(\mathcal{F}_i,{\h}_s^i)$ of dimension $n_i$, $u=||\varphi||^2$, $\cal I $ an interval, and under this isometry $\W$ is given by \eqref{www}. Moreover, for each $i=1,...,k$, each $h^i_s$ is a family of special holonomy metrics to which exactly one of the cases of Table \ref{table1} applies.
\item\label{imks2} If $(\M,\g)$ is simply connected and the vector field $\frac{1}{u_{\varphi}^2} U_{\varphi}$ is complete, the  isometry \eqref{dataf} is global with $\cal I = \mathbb{R}$. 
\item \label{imks3} Conversely, every Riemannian manifold $(\M,\g)$ of the form \eqref{dataf} with $\cal I \in \{S^1,\mathbb{R} \}$, where $u$ is any positive function and $(\mathcal{F}_i,\h^i_s)$ are families of special holonomy metrics subject to the flow equations in Table \ref{table1} is spin and admits an imaginary $\W$-Killing spinor $\varphi$ for $\W$ given by \eqref{www} with $u=||\varphi||^2$. $\varphi$ solves equation (\ref{spinor-2}).
\end{enumerate}
\end{mthm}

As a second application,
we give a local normal form for Lorentzian metrics admitting a parallel {\em null} spinor fields. To this end, one uses a relation between spinor fields and vector fields on Lorentzian manifolds provided by the {\em Dirac current}: for any spinor field $\phi$ on a Lorentzian manifold $(\bM,\bg)$ its Dirac current $V_{\phi}$ is given by 
\begin{align}
\bg(X,V_{\phi})=-\langle X\cdot \phi,\phi\rangle \label{diracc}.
\end{align} 
The zeroes of $V_{\phi}$ and $\phi$ coincide and if $\phi$ is parallel then so is $V_{\phi}$. We say that $\varphi$ is {\em null} if  $V_{\varphi}$ is a null vector. We show:

\begin{mthm} \label{nformsp}
Let $(\bM,\bg)$ be a Lorentzian manifold admitting a parallel null spinor field. Then $(\bM,\bg)$ is locally isometric to
\begin{align}
(\bM,\bg) \cong(\mathbb{R} \times \mathbb{R} \times \mathcal{F}_1 \times...\times \mathcal{F}_m,\text{ }2dvdw + \h^1_w+...+\h_w^m), \label{model}
\end{align}
for some integer $m$, manifolds $\mathcal{F}_{i}$ for $i=1,...,m$ where  each  $\h^i_w$ is a $w$-dependent family of Riemannian metrics on $\mathcal{F}_i$ to which exactly one of the cases in Table \ref{table1} applies. Conversely, every manifold as in (\ref{model}) satisfying these conditions admits a parallel null spinor.
\end{mthm}
Note that the normal forms in Theorem \ref{nformsp} need not be the most general ones. For example, in signature $(10,1)$, in \cite{bryant00} it is shown that a term $H_w dw^2$, where $H$ is an arbitrary function not depending on $v$ can be added to \eqref{model}. However, the analysis of normal forms for metrics with parallel spinor in \cite{bryant00} rests on the known orbit structure of the action of $\mathbf{Spin}(1,n)$ in for small $n$,  whereas Theorem \ref{nformsp} covers all dimensions.

\bigskip

This paper is organised as follows.
In Section \ref{prel} we recall and collect basic formulas and invariants related to the geometry of spacelike hypersurfaces in Lorentzian manifolds. 
Together with the local existence and uniqueness theorem for solutions to quasilinear first order symmetric hyperbolic systems they are the key ingredients for the {\em proof of Theorem \ref{cauchy-vf-theo}}, which  occupies a large part of the paper and consists of three main steps:
\begin{enumerate}
\item \label{step1}
In Section \ref{sec3} we derive  a first order quasilinear symmetric hyperbolic system with solutions $(\bg,\alpha,Z)$, where $\alpha\in \Omega^*(\bM)$ is a differential form, $\bg$ is a Lorentzian metric and   $\Z$ a symmetric bilinear form on $\bM$. 
\item \label{step2} As a result we will obtain a vector field $V$ and a $1$-form $E$  on $\bM$ such that $\bRic = Z - \text{Sym}(\bnab E)$, where $\bRic$ is the Ricci tensor and $\bnab$ the Levi-Civita connection of $\bg$. In Section \ref{sec4} we will then derive a  a wave equation for $E$ and $\bnab V$ and determine suitable initial conditions for all data along $\M$. Using the constraint equations (\ref{cons}) this will imply that that $E$ and $\bnab V$ vanish on $\bM$ with the conclusion that $V$ is parallel.
\item Since the solutions obtained in Steps (\ref{step1}) and (\ref{step2}) are only local,
in Section \ref{sec5} we will show how to obtain a globally hyperbolic Lorentzian manifold $(\bM,\bg)$ from these local solutions.  
\end{enumerate}
Then in Section \ref{constraintsec} we study Riemannian manifolds satisfying the constraint (\ref{cons}) and prove Theorem \ref{pf}. The proof of Theorem \ref{doublecone} is given in Section \ref{holsec} where we also study the relation to Lorentzian holonomy. Using this, the two applications in Theorems  \ref{imks} and \ref{nformsp}  are obtained in Section \ref{igr}.
 
 \subsection*{Acknowledgements} We would like to thank Helga Baum and Vicente Cort\'{e}s for inspiring discussions. TL  would like to thank   Nick Buchdahl, Mike Eastwood, Jason Lotay and Spiro Karigiannis for helpful discussions in regards to the open problem in Section~\ref{families}. AL would like to thank  Todd~Oliynyk  for discussions about symmetric hyperbolic systems during his visit to Monash University.

\section{Preliminaries} \label{prel}
 Let $(\bM,\bg)$ be a time-oriented Lorentzian manifold of dimension $(n+1)$ with global unit timelike vector field $T \in \mathfrak{X}(\bM)$. Let us now additionally assume that  $\M \subset \bM$ is a spacelike hypersurface with induced Riemannian metric $\g$ and that $T$ restricts to the future-directed unit normal vector field along $\M$.  
 We will use the following index conventions:
  \begin{itemize}
  \item Latin indices $i,j,k,...$ run from $1$ to $n$.
  \item Greek indices $\mu, \nu, \rho, ...$  run from  $0$ to $n$. We will use Greek indices
   whenever we restrict ourselves to local considerations on  $\bM$, which is topologically an open neighbourhood of $\M$ in $\mathbb{R} \times \M$. In this situation we may  fix adapted coordinates $(x^0=t,x^1,...,x^n)$, where the $t$-coordinate refers to the $\mathbb{R}$-factor, the  
Greek indices $\mu, \nu,...$ then refer to the coordinates $(x^0,...,x^n)$ on $\bM$, and Latin indices $i,j,k,...$  to the spatial coordinates $(x^1,...,x^n)$ on $\M$. We may also use this index convention when fixing a local orthonormal frame $(T=s_0,s_1, \ldots , s_n)$ with $\bg(s_\mu,s_\nu)=\epsilon_\mu\delta_{\mu\nu}$ and  $s_i\in T\M$.  It will be clear from the context whether the indices refer to coordinates or an orthonormal frame. 
   \item We will use indices $a,b,c, \ldots $ as {\em abstract indices}, i.e., only indicating the valence of  a tensor. For example, a vector field $B$ is denoted by $B^a$ and a $1$-form by $B_a$. We will however abuse this abstract index notation slightly, a when writing $0$ for a contraction $B(T, \ldots) $ of a tensor $B$ with the time like unit vector field $T$,
\[B_{0b\ldots}:=T^a B_{ab\ldots},\]
but also when using indices $i,j,k,\ldots=1, \ldots, n$ for referring to directions in $T\M$. 
\item We raise and lower indices with respect to a metric. Sometimes we also use the musical notation $\flat$ and $\sharp$ for the dualising a tensor with a metric. It will be clear from the context with which metric we are working. Throughout the paper, we will also use Einstein summation convention, i.e., summing over the same upper and lower index.
\end{itemize}
 By $\bnab$ we denote the Levi Civita connection of $\bg$. 
Moreover, $\delta = \delta^{\bg}$ denotes the divergence operator, i.e., given a $(p,0)$ tensor field $B$ on $(\bM,\bg)$, the divergence is the $(p-1,0)$-tensor
\[\delta B = - \sum_{\mu=0}^n \epsilon_\mu \left(\bnab_{s_\mu}B\right)(s_\mu,\ldots),\]
with an orthonormal basis $s_\mu$, 
or with abstract index notation \[(\delta B)_{b\ldots c}=-\nabla_a B^a_{~b\ldots c}.\] For a vector field $V$ we have $\text{div}^{\bg} V = - \delta V^{\flat}$, or in indices $\div^\bg(V)= \nabla_aV^a$.
For $X,Y\in T\M$ we denote by
\[
\W(X,Y):=-\g( \bnab_XT,Y)
\]
the second fundamental form of $(\M,\g)\subset (\bM,\bg)$, i.e.,
we have
\begin{align}
\bnab_XY=\nabla_XY-\W(X,Y) T,\label{lcf}
\end{align}
where barred objects refer to data on $\bM$ and unbarred objects to data on $\M$. 
The dual of the  second fundamental form is the (symmetric) {\em Weingarten operator}, also denoted by $\W$, i.e., $\W(X,Y)=\g(\W(X),Y)$.
It holds that $\W=-\bnab T|_{T\M}$.
The curvature tensors of $(\bM, \bg)$ and $(\M,\g)$ are related via the Gau{\ss},  Codazzi and Mainardi equation. Here we need the following contracted version: let 
\[ \overline{G} = \bRic - \tfrac{1}{2}\overline{\mathrm{scal}} \cdot \bg \]
denote the divergence-free Einstein tensor of $(\bM,\bg)$, where $\bRic$ is the Ricci tensor and $\bscal$ the scalar curvature of $\bg$. Then we have on $\M$:
\begin{equation}
\begin{aligned} \label{hsf}
\overline{G}(T,T) & = \tfrac{1}{2} \left(\mathrm{scal}^{\g} -\mathrm{tr}_{\g}(\W^2) +  (\mathrm{tr}_{\g}\W)^2 \right),\\
\overline{G}(T,X) & = (\delta^{\g} \W) (X)  + d(\mathrm{tr}_{\g}\W)(X), \text{ for all  } X \in T\M.
\end{aligned}
\end{equation}
Now we specialise the discussion to the case that $\bM$ is topologically an open subset of $\mathbb{R} \times \M$ for some manifold $\M$. We assume that $\partial_t$ is timelike everywhere on $\bM$ and set \[T=\tfrac{1}{\sqrt{\bg (\partial_t,\partial_t)}} \partial_t\] and $\bg$ restricted to $T\M$ is then positive definite. Note that writing $T\M$ in this context refers more precisely to the pullback bundle $\pi^*T\M \rightarrow \bM$, where $\pi :\bM \rightarrow \M \cong \{0\} \times \M$ denotes the projection.

Next, suppose that $V \in \mathfrak{X}(\bM)$ is a {\em null vector field} on $\bM$, i.e. a non-vanishing smooth vector field $V$ such that  $\bg(V,V) = 0$. We decompose $V$ with respect to the splitting $T\bM = \mathbb{R} \partial_t \oplus T\M$, which need not be $\bg$-orthogonal, into
\begin{align}
V = u T - U = u \cdot( T- N), \label{split}
\end{align}
where $u =-\bg(V,T)\in C^{\infty}(\bM)$ a non-vanishing function, $U \in \Gamma(\pi^*T\M)$, $u^2 = \bg(U,U)$ and $N = \frac{1}{u} U$. We also write $N_t$ in order to emphasise the $t$-dependence. Note that $V \neq 0$ and $\bg(V,V) = 0$ requires that $v, u$ and $U$ do not vanish at any point. We emphasise that $\bg(T,U)$ is not necessarily zero on $\bM$. However, we have that $\bg(T,U)_{|\{0 \} \times \M} = 0$ as $T$ was assumed to be the unit normal vector field along $\M$. We identify $\M$ within $\bM$ as $\{0 \} \times \M$. It follows that $v|_{\M} = u|_{\M}$.

Finally we assume that the null vector field $V$ is {\em parallel}, i.e., that $\bnab V \equiv 0$. As a consequence of \eqref{lcf} we obtain for every $X \in T\M$
\begin{equation}
 \label{ij1}
0 \ =\  \text{pr}_{T\M} (\bnab_X V)|_{\M} 
\ =\  -u|_{\M} \W(X) - \nabla_X U 
\ =\ -u|_{\M} \W(X) - \nabla_X U,
\end{equation}
which is precisely the constraint equation \eqref{cons}. 

For the proof of Theorem \ref{cauchy-vf-theo} we will have to analyze various PDEs. As it turns out, they can all be locally reduced to a {\em first order quasilinear symmetric hyperbolic system}. We collect some standard facts on that: 

Consider an equation of the form
\begin{align}
 A^0(t,x,w) \partial_t w =  A^{i}(t,x,w) \partial_{i} w + b(t,x,w), \label{symmh}
\end{align} 
for $k$ real functions on $\rr\times \rr^n$ which are collected in a vector valued function $w(t,x) \in \mathbb{R}^k$. The solution will be defined on an appropriate subset of $\mathbb{R} \times \mathbb{R}^n$ and $(t,x)$ denotes a point in $\mathbb{R} \times \mathbb{R}^n$. Equation~\eqref{symmh} is called \textit{quasilinear symmetric hyperbolic} if the matrices $A^0$ and $A^i$, which may depend on the point $(t,x)$ as well as on the unknown $w$ itself, are symmetric and $A^0$ is positive definite. For given smooth initial data and smooth coefficients, there is a well established local existence and uniqueness result for smooth solutions $w$ to \eqref{symmh} which we shall use repeatedly. For details we refer to \cite[Section 16]{Taylor11-3} or \cite{FriedrichRendall00} and references therein.

\section{Proof of Theorem \ref{cauchy-vf-theo}: the quasilinear symmetric hyperbolic system} \label{sec3}

As indicated in the introduction, for clarity the proof is subdivided into various steps: first, we find evolution equations whose solutions define the metric $\bg$ and the vector field $V$ locally; then we show that the constructed vector field $V$ is indeed parallel and patch the locally defined solutions together and discuss global properties.
In this section we deal with the evolution equations.

\subsection{Finding the evolution equations}
In order to get an idea of how to obtain the desired metric $\bg$ and the vector field $V$ from the data $(\M,\g,U)$, suppose for a moment that we already have a Lorentzian manifold $(\bM,\bg)$ such that
\begin{itemize}
\item $\bM$ is an open subset of $\M \cong \{0 \} \times \M$ in $\mathbb{R} \times \M$,
\item $\widetilde{\lambda}^2:=-\bg(\partial_t,\partial_t) > 0$, where $\partial_t=\partial_0 $ refers to the vector field corresponding to the $t$-coordinate. This defines a time-like unit vector field $T=\widetilde{\lambda}^{-1}\partial_t$.
\item There exists a  parallel null vector field $V \in \mathfrak{X}(\bM)$.  This defines a space like  unit vector field $N$ by relation (\ref{split}), i.e., by $V=v(T-N)$.
\end{itemize} 
We derive some evolution equations as consequences:

 Let $\alpha = V^{\flat}$ denote the $\bg$-dual 1-form to $V$ and consider $\alpha$ as a section in the exterior algebra,  i.e.,  $\alpha \in\Omega^*(\bM)$. Since  $V$ is parallel, $\alpha$ is parallel and we have 
\begin{align}
(d+\delta^{\bg}) \alpha = 0, \label{eva}
\end{align}
where 
 $d+\delta^{\bg} : \Omega^*(\bM) \rightarrow \Omega^*(\bM)$ is  the de Rham operator. It is  given by \[d+\delta = c \circ \nabla,\] where $c: T\bM \otimes \Lambda^* \bM \rightarrow \Lambda^* \bM$ denotes Clifford multiplication by forms, i.e.,
\[ c(X) \omega = X^{\flat} \wedge \omega - \iota_X \omega, \text{ for all  } X \in T\bM, \]
where $\iota_X\omega=\omega(X, \ldots)$ denotes the interior product. 
Symbolically this can be written as
\[ c(X) = (X^{\flat} \wedge)  - \iota_X , \text{ for all  } X \in T\bM. \]
The de Rham operator in \eqref{eva} is of Dirac type, which suggest that it is hyperbolic. We will explicitly verify this later.

Next, as $V$ is parallel, it annihilates the curvature tensor $\bR$ of $\bnab$, i.e., $\bR(V,.,.,.)=0$. In particular, 
\begin{align}
\bRic(V,\cdot) = 0. \label{4r}
\end{align}
To evaluate this further, note that the metric $\bg$ defines a (non orthogonal) splitting 
\begin{align}
T\bM = T\M \oplus \mathbb{R} V \label{te}
\end{align}
 of bundles over $\bM$. We introduce the $\bg$- and $V$-dependent projection
\begin{align*}
\pr^{\bg,V}_{T\M} : T\bM \rightarrow T\M
\end{align*} 
onto the first factor in the splitting  \eqref{te}. That this projection is dependent on $\bg$ and $V$ becomes evident when it is written as
\[ \pr^{\bg,V}_{T\M}=\mathrm{Id}_{T\bM}+ \tfrac{1}{u}\bg(T,.)V,\]
or, written in in local coordinates $(x^0=t, x^1, \ldots, x^n)$ with $(x^1, \ldots, x^n)$ coordinates on $\M$, 
\[
\left(\pr^{\bg,V}_{T\M}\right)_\mu^{\ \ \nu}=\delta_\mu^{~\, \nu}- \tfrac{1}{\widetilde{\lambda}u} \bg_{0\mu}V^\nu.\]
 Note that the endomorphism $\pr^{\bg,V}_{T\M}$ is constant in direction of $V$, 
\begin{equation}\label{nabvpv}
\bnab_V \pr^{\bg,V}_{T\M}=0.\end{equation}
 In contrast to this, we denote by $\mathrm{pr}_{T\M}$ the standard and $\bg$-independent projection
\begin{align*}
\mathrm{pr}_{T\M} : T\bM = \mathbb{R}\partial_t \oplus T\M \rightarrow T\M
\end{align*}
onto the second factor. But then equation \eqref{4r} is equivalent to
\begin{align}
\bRic = \Z \circ \pr^{\bg,V}_{T\M}, \label{evb}
\end{align}
where $\Z$ is a symmetric bilinear form on $T\M$, i.e. $Z \in \Gamma(\pi^*(T^*\M \otimes T^*\M))$, which is trivially extended to a symmetric bilinear form on $T\bM = T\M \oplus \mathbb{R} V$.  Again in local coordinates, this equation becomes
\[\bRic_{\mu\nu}=\Z_{\mu\nu}- \tfrac{2}{\widetilde{\lambda}u} V^\rho \Z_{\rho(\mu}\bg_{\nu)0} +
\tfrac{1}{(\widetilde{\lambda}u)^2}V^\rho V^\sigma\Z_{\rho\sigma}\bg_{0\mu}\bg_{0\nu}.\]
Finally, a first order equation for $\Z$ is then derived as follows. 
As every expression of the form $\bR(V,\cdot,\cdot,\cdot,\cdot)$ vanishes identically on $\bM$, it follows from the second Bianchi identity that $\bnab_V \bRic = 0$, which in particular implies that
\begin{align}
\bnab_V \Z = 0. \label{evc}
\end{align}
Seeking  for a different formulation of this condition,  we use the splitting \eqref{split} of $V$ into  $T$ and $N$, both depending also on $t$, to see that \eqref{evc}  becomes 
\begin{align} \label{zur1}
0 =
(\bnab_{\partial_t} \Z)-\widetilde{\lambda} \  \bnab_{N} \Z.
\end{align}
Now, for brevity we rewrite this condition in local coordinates $(x^0=t, x^1, \ldots, x^n)$ with $(x^1, \ldots, x^n)$ coordinates on $\M$. We obtain that equation (\ref{zur}) is equivalent to 
 \begin{equation}
 \label{ze}
\partial_t\Z_{kl}=\widetilde{\lambda}N^i\partial_i\Z_{kl} +2\overline{\Gamma}^i_{0(k}\Z_{l)i}- 2 \widetilde{\lambda}N^i\overline{\Gamma}^j_{i(k}\Z_{l)j},\end{equation}
in which $\widetilde{\lambda}=\sqrt{-\bg(\partial_t,\partial_t)}>0$, the unit vector field $N$ depends on $V$ via relation (\ref{split}),  and the round brackets denote the symmetrisation of indices.

The advantage of this formulation is that \eqref{ze} is manifestly a $\bg$- and $\partial \bg$-dependent $t$-evolution equation for a $t$-dependent family of symmetric endomorphisms $\Z_t \in \Gamma(\M,T^*\M \otimes T^*\M)$ on $\M$. 
\subsection{Hyperbolic reduction}
Let $(\M,\g)$ be a Riemannian manifold and let $(U,\W)$ be a nontrivial solution to \eqref{cons}. The idea is to impose equations \eqref{eva}, \eqref{evb} and \eqref{ze} locally as a coupled PDE system of first order evolution equations for the unknowns $w=(\alpha,\bg,\partial \bg, \Z)$
 defined on a neighbourhood of $\M$ in $\mathbb{R} \times \M$ with initial data to be specified. 

More precisely, we would like to rewrite \eqref{eva}, \eqref{evb} and \eqref{ze} 
locally as a {\em first} order quasilinear symmetric hyperbolic PDE of the form \eqref{symmh}. A well studied  technical problem is that $\bRic$ is not hyperbolic when being considered as differential operator acting on the metric. There is a standard tool used for the Cauchy problem for the Einstein equations in general relativity how to overcome this, which is referred to as {\em hyperbolic reduction} and explained in detail in \cite{Ringstrom09}. To this end, we bring into play a fixed background metric
\begin{align}
  h:= -\lambda^2 dt^2 + \g. \label{hlam}
\end{align}
on $\mathbb{R} \times \M$, where $\lambda$ is the prescribed function from Theorem \ref{cauchy-vf-theo}. Given local coordinates $(x^0,\ldots, x^n)$, we denote by $\widetilde{\Gamma}^{\mu}_{\alpha \beta}$ the Christoffel symbols of $h$. For any metric $\bg$ on $\mathbb{R} \times \M$ with Christoffel symbols $\Gamma_{\alpha\beta}^\mu$ we then introduce the difference tensor $A^{\mu}_{\alpha \beta} = \Gamma^{\mu}_{\alpha \beta} - \widetilde{\Gamma}^{\mu}_{\alpha \beta}$ and let 
\begin{equation}
\begin{aligned}\label{humm}
F_{\nu} &= g_{\mu \nu} g^{\alpha \beta} \widetilde{\Gamma}^{\mu}_{\alpha \beta}, \\
E_{\nu} &=-g_{\mu \nu}g^{\alpha \beta}A^{\mu}_{\alpha \beta}. 
\end{aligned}
\end{equation}
We denote by $\text{Sym}(\bnab E)[\bg]$ the symmetrisation of the $(2,0)$-tensor $\bg(\bnab E, \cdot)$ for any given Lorentzian metric $\bg$, i.e., $\text{Sym}(\bnab E)(X,Y)=\tfrac{1}{2} \left( (\bnab_XE)(Y)+ (\bnab_YE)(X)\right)$.  Then the operator
\[ \widehat{\Ric}[\bg]:= \Ric[\bg] + \text{Sym}(\bnab E)[\bg] \]
is in coordinates given by

\begin{align} 
 \widehat{\Ric}_{\mu \nu} = - \tfrac{1}{2}\bg^{\alpha \beta}\partial_{\alpha} \partial_{\beta} \bg_{\mu \nu} + \nabla_{(\mu}F_{\nu)} + \underbrace{\bg^{\alpha \beta} \bg^{\gamma \delta} [ \Gamma_{\alpha \gamma \mu} \Gamma_{\beta \delta \nu} + \Gamma_{\alpha \gamma \mu} \Gamma_{\beta \nu \delta} + \Gamma_{\alpha \gamma \nu} \Gamma_{\beta \mu \delta} ]}_{=:H_{\mu \nu}[\bg,\partial \bg]},
\end{align}
where we use the standard notation for symmetrisation $\bnab_{(\mu}F_{\nu)}=\tfrac{1}{2}\left(\bnab_{\mu}F_{\nu}+\bnab_{\nu}F_{\mu}\right)$.
The crucial point is that second order derivatives of $\bg$ appear only in the first term of $\widehat{\Ric}[\bg]$ (assured by addition of $E$, $F$ depends only on $\bg$ and not on its derivatives). 
Hence, in the following we will replace equation \eqref{evb} by the equation
 \begin{align}
\bRic = Z \circ \pr^{\bg,V}_{T\M} - \text{Sym}(\bnab E), \label{eq11}
\end{align}
where we abbreviate the Ricci tensor of $\bg$ as  $\bRic=\Ric[\bg]$.
Of course, eventually we will construct a solution and then show that $E=0$.

\subsection{Local evolution equations as first order symmetric hyperbolic system}

After these preparations, we are now able to show:
\btheo \label{funpde}
 Under the assumptions of Theorem \ref{cauchy-vf-theo} with given data $(\M,\g)$, $\lambda$ and $U$, every point $p \in \M$ admits an open neighbourhood $\mathcal{V}_p$ in $\mathbb{R} \times \M$ on which the equations \eqref{eva}, \eqref{ze} and \eqref{eq11},
considered as coupled PDE for the unknowns $(\bg,\alpha,Z)$, are locally equivalent to a first order quasilinear symmetric hyperbolic PDE of the form \eqref{symmh} provided that \[
\bg|_{\M} = h\quad\text{ and }\quad\alpha|_{\M} = (h(\tfrac{u}{\lambda} \partial_t - U,\cdot))|_{\M},\] 
where
 $h$ is the background metric defined from $\lambda$ and $\g$  in equation \eqref{hlam}.
 \etheo
 
 \bprf
 Note first that for each choice of the unknowns $\alpha=(\alpha_0,....,\alpha_n) \in \Omega^*(\bM)$, where $\alpha_i \in \Omega^i(\bM)$ and $\bg$ we can form vector fields $V \in \mathfrak{X}(\bM)$, $U \in \mathfrak{X}(\bM)$ algebraically by $V=V[\alpha,\bg] = \alpha_1^{~\sharp}$ and $U=U[\alpha,\bg] = -\mathrm{pr}_{T\M}V$. Moreover, we set $u^2=u^2[\alpha,\bg] = \bg(U,U)$. Now fix $p \in \M$ and choose $\mathcal{V}_p$ to be a coordinate neighbourhood of $p \in \mathbb{R} \times \M$ with coordinates $(x^0=t,x^1,...,x^n)$. We define the following open subset in the space of Lorentzian metrics on $\mathcal{V}_p$:
\begin{align}
\mathcal{G}_p := \{ \bg \mid \bg(\partial_0,\partial_0) < 0,\text{ } dt(\text{grad}^{\bg}t) < 0,\text{  }\bg_{|T\M \otimes T\M} > 0 \}
\end{align}
Note that $h \in \mathcal{G}_p$. Given any metric $\bg \in \mathcal{G}_p$, we fix a $\bg$-dependent pseudo-orthonormal basis $(s_0,...,s_n)$ for $\bg$, i.e. $\bg(s_a,s_b) = \epsilon_a \delta_{ab}$, by applying the Gram Schmidt procedure to $(\partial_t,\partial_1,...,\partial_n)$. That is, $s_0 = \frac{1}{\sqrt{(-\bg{(\partial_t,\partial_t)})}} \partial_t=T$ and for $i>0$
\begin{align}
  s_i = s_i[\bg] = \sum_{\mu=0}^n \zeta^{\mu}_i[\bg] \partial_{\mu} \label{rea}
\end{align}
on $\mathcal{V}_p$ for certain coefficients $\zeta^{\mu}_i[\g]$ which depend smoothly and only algebraically on $\bg$. Note that choosing $\bg$ from $\mathcal{G}_p$ ensures that the Gram-Schmidt algorithm is well defined for $(\partial_t,\partial_1,...,\partial_n)$. By the special form of the fixed background metric $h=-\lambda^2 dt^2 + \g$ we have that $\zeta^0_{i>0}[h]=0$. For any $\bg \in \mathcal{G}_p$ we then rewrite equations \eqref{eva}, \eqref{ze} and \eqref{eq11} on $\mathcal{V}_p$ as follows: \\

\subsubsection*{Local reformulation of equation (\ref{eq11})} 
In analogy to \cite{FischerMarsden72-1}, for any Lorentzian metric $\bg \in \mathcal{G}_p$ and  quantities $k_{\mu \nu}$ and $\bg_{\mu \nu, i}$ we consider the system 
\begin{align}
\partial_t \bg_{\mu \nu} &= k_{\mu \nu}, \label{f1} \\
\bg^{ij} \partial_t \bg_{\mu \nu, i} &= \bg^{ij} \partial_{i} k_{\mu \nu}, \label{f2} \\
- \bg^{00} \partial_t k_{\mu \nu} &= 2 \bg^{0 j} \partial_{j} k_{\mu \nu} + \bg^{ij} \partial_{j} \bg_{\mu \nu, i} - 2 H_{\mu \nu}[\bg,k] - 2 \nabla_{(\mu}F_{\nu)}[\bg,k] + 2 \, (\Z \circ \pr^{\bg,V}_{T\M})_{\mu \nu}, \label{f3}
\end{align}
with initial conditions $\bg|_\M=h|_\M$ and 
\begin{equation} \bg_{\mu \nu, i}|_{t=0}=\partial_{i} \bg_{\mu \nu}|_{t=0}= \partial_i h_{\mu\nu}|_{\t=0}
\label{gmni-ini}
\end{equation}
This system with the given initial condition  is equivalent to equation \eqref{eq11}\footnote{This has been shown in \cite{FischerMarsden72-1} for the vacuum Einstein equations $\bRic = 0$  and remains valid in our setting, as here the $\Z$-term in \eqref{eq11} enters only algebraically in the $b_1$-term.}. Indeed, let a triple $(\bg_{\mu \nu},k_{\mu \nu},\bg_{\mu\nu,i})$ solve system \eqref{f1}-\eqref{gmni-ini}. As $\bg^{ij}$ is invertible for $\bg$ sufficiently close to $h$, equation \eqref{f2} is the same as $\partial_t \bg_{\mu \nu, i}=\partial_{i} k_{\mu \nu}$, and equation \eqref{f1} then gives 
\[\partial_t (\bg_{\mu \nu, i}-\partial_{i} \bg_{\mu \nu})=0.\] 
Initial  condition \eqref{gmni-ini}
 ensures $\bg_{\mu \nu, i}-\partial_{i} \bg_{\mu \nu}=0$ at $t=0$ and thus  everywhere. Then equation \eqref{f3} is nothing but equation \eqref{eq11}. Hence, for any fixed $\Z$, the system \eqref{f1}-\eqref{f3} can be rewritten as
\begin{align} \label{co1}
A^0_1(t,x,w^1) \partial_0 w^1 =  A^i_1(t,x,w^1) \partial_i w^1 + b_1(t,x,w^1,\Z,\alpha)
\end{align}
where $w_1 = (\bg_{\mu \nu}, (\bg_{\mu \nu,i})_{i=1,...,n}, k_{\mu \nu})_{\mu, \nu = 0,...,n}$. Moreover, the matrices $A^0_1$ and $A^i_1$ are symmetric and $A^0_1(t,x,w_1)$ is positive definite for $\bg = h$, and hence in a neighbourhood of $h$. In fact, they can be written as 
\[A_1^0
=
\begin{pmatrix}
1 &0&0
\\
0& g^{00} & 0
\\
0&0&-g_i^{~j}
\end{pmatrix}
,
\qquad
A_1^i
=
\begin{pmatrix}
0 &0&0
\\
0& 2g^{0i} & g^{ij}
\\
0&g^{ij}&0
\end{pmatrix}
.\]

\subsubsection*{Local reformulation of equation (\ref{eva})}
Using the orthonormal basis $s_\mu$, we can identify $\bnab\alpha \in T^*M\otimes \Omega^*(M)$ with
\[
- s_0\otimes \bnab_{s_0}\alpha +\sum_{k=1}^ns_k\otimes \bnab_{s_k}\alpha\ \in \ TM\otimes \Omega^*(M).\]
With this identification, equation \eqref{eva} writes as 
\[
0= -c(s_0)\bnab_{s_0}\alpha +\sum_{k=1}^n c(s_k)\bnab_{s_k}\alpha.\]
Using the fundamental Clifford identity \[c(X) \circ c(Y) + c(Y) \circ c(X) = -2 \bg(X,Y) \cdot 1\] for the $\bg$-dependent operator $c$,  equation \eqref{eva} for $\bg \in \mathcal{G}_p$  is equivalent to
\begin{eqnarray*}
\tfrac{1}{\widetilde{\lambda}} \bnab_{\partial_t} \alpha
&=&
 \sum_{k=1}^n
c(s_0) \circ c(s_k)\bnab_{s_a}\alpha
\\
&=&
 \sum_{k=1}^n  \zeta^{0}_k[\bg] \,
c(s_0) \circ c(s_k)\bnab_{\partial_t}\alpha
+  \sum_{i,k=1}^n   \zeta^{i}_k[\bg] \,
c(s_0) \circ c(s_k)\bnab_{\partial_i}\alpha,
\end{eqnarray*}
which can be re arranged to 
\begin{align} \label{26}
\left(\tfrac{1}{\widetilde{\lambda}} -
  \sum_{k=1}^n \zeta^{0}_k[\bg] \, c(s_0)\circ c(s_k) \right) \bnab_{\partial_t} \alpha= \sum_{i,k=1}^n \zeta^{i}_k[\bg] \,c(s_0)\circ c(s_k) \, \bnab_{\partial_i} \alpha.
\end{align}
%
%
By means of the fixed coordinates, we identify $\alpha$ with a smooth map $\alpha:\mathcal{V}_p \rightarrow \Lambda^* \mathbb{R}^{n+1} \cong \mathbb{R}^{ 2^{n+1}}$. In this identification, $\bnab_{\partial_{\mu}} = \partial_{\mu} + \Gamma$, for an endomorphism $\Gamma$ which depends on the Christoffel symbols of $\bg$. Then equation  \eqref{26} becomes  equivalent to a system
\begin{align} \label{co2}
A_2^0(t,x,\bg,\alpha) \partial_t \alpha = \sum_{i = 1}^n A_2^{i}(t,x,\bg,\alpha) \partial_{i} \alpha + b_2(t,x,\alpha,\bg,\partial \bg) 
\end{align}
We claim that the matrices $A_2^0$ and $A_2^i$ are symmetric. To this end, let $(e_0,...,e_n)$ denote the standard basis of $\mathbb{R}^{n+1}$. We consider the operator $c(e_\mu) = (e_\mu^{\flat} \wedge) - \iota_{e_\mu}$, where the dual is formed using the standard Minkowski inner product on $\mathbb{R}^{n+1}$.
Now let $\sigma^\mu$ be the (algebraically) dual basis to $e_\mu$, i.e., with $\sigma^\mu(e_\nu)=\delta^\mu_{\ \nu}$ and with 
$e_\mu^\flat=\epsilon_\mu \sigma^\mu$, where $\epsilon_0 = -1$, $\epsilon_{i>0} = 1$. Furthermore, let
 $\langle \cdot , \cdot \rangle $ be the standard \textit{positive definite} inner product on $\Lambda^* \mathbb{R}^{n+1}$, i.e., with $\sigma^\mu, \sigma^\mu\wedge \sigma^\nu, \ldots, \sigma^0\wedge\ldots \wedge \sigma^n $ as orthonormal basis.
Then
elementary linear algebra shows that
\[ \left\langle c(e_\mu) \gamma, \delta \right\rangle = -\epsilon_\mu \left\langle \gamma, c(e_\mu) \delta \right\rangle, \text{ for all  } \gamma, \delta \in \Lambda^* \mathbb{R}^{n+1}. \]
It follows from the Clifford identity for $c$ that for $i>0$
\[ \left\langle (c(e_0) \circ c(e_i)) \gamma, \delta \right\rangle =  \left\langle \gamma, (c(e_0) \circ c(e_i)) \delta \right\rangle, \]
which proves symmetry of the linear map $c(e_0)\circ c(e_\mu)$ and hence of the matrices $A^{\mu}_2$. Moreover, for $\bg = h$, $A^0_2(t,x,h)$ reduces to a positive multiple of the identity matrix. Thus, $A^0_2$ is positive definite in a neighbourhood $\mathcal{V}_p$ of the initial data if these initial data are chosen as required in the theorem.

\subsubsection*{Local reformulation of \eqref{ze}}
Locally, the $t$-dependent symmetric bilinear form $\Z$ on $T\M$ can be rewritten as $\Z = \Z_{kl} dx^k dx^l$ for $t$- and $x$ dependent coefficients $Z_{kl}$. One verifies immediately that \eqref{ze} is of the form
\begin{equation} \label{co3}
A_3^0(t,x) \partial_t (\Z_{kl})_{k,l>0} = \sum_{i = 1}^n A_3^{i}(t,x,\bg,\alpha) \partial_{i} (\Z_{kl})_{k,l>0} + b_3(t,x,\Z,\bg,\partial \bg, \alpha) 
\end{equation}
where $A_3^0(t,x)=\mathrm{Id}$ is simply the identity matrix and $A_3^{i}(t,x,u)$ are multiples of the identity matrix. 

Combining \eqref{co1}, \eqref{co2} and \eqref{co3} gives a coupled PDE of the form \eqref{symmh} with matrices $A^0$ and $A^i$ being block diagonal with blocks $A_1^0$, $A_2^0$ and $A_3^0$, and blocks $A_1^i$, $A_2^i$ and $A_3^i$, respectively.
The unknowns are  $w=(w_1,w_2,w_3)$, with $w_1 = (\bg_{\mu \nu}, (\bg_{\mu \nu,i}), k_{\mu \nu})$, $w_2=\alpha$, $w_3=\Z_{kl}$,  and the inhomogeneity is $b=(b_1,b_2,b_3)$, which is defined in a neighbourhood of the initial data. Moreover, the previous discussion regarding the blocks of  $A^0$ and the $A^i$'s shows that $A^0$ and $A^i$ are symmetric and $A^0$ is positive definite at least in a local neighbourhood of the initial data.
 \eprf

\subsection{Initial data} \label{idat}
Here we specify a {\em full set of initial data} for the first order PDE for the quantities $(\bg_{\mu \nu}, \bg_{\mu \nu,i}, k_{ij}, \alpha, Z_{kl})$ on $\mathcal{V}_p$ derived in Theorem \ref{funpde}. Initial data for $\bg$ and $\alpha$ were already given in Theorem \ref{funpde} and are needed to ensure that the PDE is indeed hyperbolic. Moreover,  as seen in the proof of Theorem \ref{funpde}, 
to ensure that the system \eqref{f1}-\eqref{f3} is equivalent to \eqref{eq11}, we were forced to 
set 
\[ {\bg_{\mu \nu, i}}_{|t=0}={\partial_{i} \bg_{\mu \nu}}_{|t=0}\]
as  initial condition  $\bg_{\mu\nu, i}$.

Regarding $k_{\mu \nu}$ we observe that $\frac{1}{\lambda}\partial_t$ is the unit normal vector field with respect to $h$ along $\M$ and set 
\begin{align}
{k_{ij}}_{|t=0} = -2 \lambda|_{\M}  \W(\partial_i,\partial_j), \label{umms}
\end{align}
This is 
required, of course, by the fact that $(\M,\g)$ should eventually embed into the solution $(\bM,\bg)$ with Weingarten tensor being the given $\W$. The initial data for $k_{i 0}$ and $k_{00}$ are uniquely determined by the natural requirement \[{(E_{\mu})}_{|t=0} {=} 0\]  for any solution $\bg$. It is by definition of $E$ straightforward to compute, see \cite{Ringstrom09}, that  this is the case if and only if 
\begin{equation}
\begin{aligned} \label{drop}
 {k_{00}}_{|t=0} & = -2 \lambda|_{\M}^2 \, {F_0}_{|t=0} + 2 \lambda|_{\M}^3 \,\mathrm{tr}_{\g} \W, \\
 {k_{0 i}}_{|t=0} & = \lambda|_{\M}^2 \, \left[-F_{i} + \tfrac{1}{2} \g^{jk} (2 \partial_{j} \g_{ki} - \partial_{i} \g_{jk}) +\partial_i(\text{log}\lambda|_{\M})\right]_{|t=0}.
\end{aligned}
\end{equation}
Note that it makes sense here to write $F_{|t=0}$, as by equations \eqref{humm} the $F$-dependence on $\bg$ is only algebraic and $\bg_{|t=0}$ has already been specified. 
Moreover, for the background metric $h$ as in~(\ref{hlam}) and initial conditions for $\bg$ as in Theorem~\ref{funpde}, the initial conditions~(\ref{drop}), simplify to
\begin{equation}
\begin{aligned} \label{dropsimple}
 {k_{00}}_{|t=0} & = -2 \lambda|_{\M}^2 \, \dot{\lambda}|_{\M} + 2 \lambda|_{\M}^3 \,\mathrm{tr}_{\g} \W, \\
 {k_{0 i}}_{|t=0} & =0.
\end{aligned}
\end{equation}
This makes evident that the  initial conditions (\ref{drop}) are independent of the chosen coordinates.

Next, we give initial data for the symmetric bilinear form $\Z|_{\M}$ on $\M$. Their origin is not very transparent at this point, but we shall see in a later step of the proof that the following initial data for $\Z$ are demanded by requiring that $\bnab V = 0$. We set
\begin{equation}
\begin{array}{rcl}  \label{iv}
\Z|_{\M}(U,\cdot) &=& u(d(\mathrm{tr}_{\g}\W)+ \delta^{\g}\W ), \\
\Z|_{\M}(X,Y) &=& \Ric(X,Y) -\R(X,N,N,Y)
 - \W^2(X,Y) + \W(X,Y)\mathrm{tr}_{\g}\W
 \\
&& + \W(X,N)\W(Y,N) 
-\W(X,Y) \W(N,N),
\end{array}
\end{equation}
for all  $X,Y \in U^{\perp}$ and 
where as usual $N=\frac{1}{u}U$.
\subsection{Solving the evolution equation}
Combining the choice of initial data with Theorem \ref{funpde} we find, using the existence and uniqueness result for symmetric hyperbolic systems as discussed earlier, a neighbourhood $\mathcal{U}_p \subset \mathcal{V}_p$ of $p$ in $\mathbb{R} \times \M$ such that the system \eqref{eva}, \eqref{ze} and \eqref{eq11} has a unique smooth solution on $\mathcal{U}_p$ which coincides on $\M \cap \mathcal{U}_p$ with the initial data.

Given this solution $\left( \bg_{\mu \nu}, \  \bg_{\mu \nu, i} ,\  k_{\mu \nu} ,\  \alpha, \ \Z_{kl} \right)$, we define with the coordinates $x^{\mu}$ on $\mathcal{V}_{p}$ specified earlier the bilinear from $\bg=\bg^{\mathcal{U}_p} = {\bg}_{\mu \nu} dx^{\mu} dx^{\nu}$ on  $\mathcal{U}_p$. Furthermore, after restricting $\mathcal{U}_p$ if necessary we may assume that $\bg$ is of Lorentzian signature on $\mathcal{U}_p$ and an element of $\mathcal{G}_p$ as this holds for the initial datum $h$. Moreover $\Z=\Z^{\mathcal{U}_p}=\Z_{kl}dx^k dx^l$ defines a symmetric bilinear form on $\mathcal{U}_p$ and the solution gives $\alpha = \alpha^{\mathcal{U}_p} \in \Omega^*{(\mathcal{U}_p)}$.

For reasons related to global hyperbolicity, which become clear in the last step of the proof, we restrict the solution domain $\mathcal{U}_p$ further as follows.
Let 
\[F^{\mathcal{U}_p}:=\frac{1}{dt(\text{grad}^{\bg}(t))}\text{grad}^{\bg}(t) \in \mathfrak{X}(\mathcal{U}_p),\] where $t$ denotes the function $(t,x) \mapsto t$ and denote by $\phi^{\mathcal{U}_p}$ the flow of $F$. We restrict $\mathcal{U}_p$ to an open neighbourhood of $p$ in $\mathbb{R} \times \M$, denoted with the same symbol, such that
\begin{align}
\forall q \in \mathcal{U}_p :  \exists \tau = \tau(q) \in \mathbb{R}: \text{ } \phi^{\mathcal{U}_p}_{-\tau}(q) \in (\{0\} \times \M) \cap \mathcal{U}_p. \label{finalstep}
\end{align}

It is possible to restrict $\mathcal{U}_p$ further (denoted by the same symbol) such that  the spacelike hypersurface $\M_p := \M \cap \mathcal{U}_p$ is a Cauchy hypersurface in $(\mathcal{U}_p,g)$, for details see \cite[Chapter A.5]{BarGinouxPfaffle07}. By construction of the initial data \eqref{umms} and as $k_{\mu \nu} = \partial_t \bg_{\mu \nu}$, $(\M_p,\g)$ embeds into $\mathcal{U}_p$ with Weingarten tensor (the restriction of) $\W$.


 \section{Proof of Theorem \ref{cauchy-vf-theo}: the wave equation}\label{sec4}
 In this section we continue the proof of Theorem \ref{cauchy-vf-theo} by deriving a linear wave equation on $E$ and $\nabla V$, the  solutions obtained in the previous section,  as well  as appropriate initial conditions that will ensure that $E=0$ and $\nabla V=0$.
\subsection{Fundamental properties of the solution}
Let $(\bg,\alpha,\Z)$ denote the local solution to the system \eqref{eva}, \eqref{ze} and \eqref{eq11}. A priori, it is not clear that $\alpha $ is a $1$-form and defines via $\bg$ a vector field. However, if we consider the Hodge-Laplacian  $\Delta^{HL} = (d+\delta)^2$ on forms and  decompose the solution $\alpha$ as $\alpha = \alpha_0 + ...+\alpha_{n+1} \in \Omega^0(\mathcal{U}_p) \oplus....\oplus \Omega^{n+1}(\mathcal{U}_p)$ we get  as a trivial consequence of $(d + \delta) \alpha = 0$ that
\begin{align} \label{hl}
\Delta^{HL} \alpha_i = 0,\text{ for all }i=0,...,n+1. 
\end{align}
Moreover, our choice of initial data and $(d+\delta)\alpha = 0$ guarantees that for $i\neq 1$
\begin{equation}
\begin{aligned} \label{ina}
(\alpha_i)_{|\M_p} &= 0, \\
(\bnab_{\partial_t} \alpha_i)_{|\M_p} &= 0,
\end{aligned}
\end{equation}
where $\M_p=\cal U_p\cap \M$ as before.
By the main result of \cite{BarGinouxPfaffle07}, the Cauchy problem for the normally hyperbolic operator $\Delta^{HL}$ is well posed and as $\M_p \subset \mathcal{U}_p$ is a Cauchy hypersurface, we conclude that $\alpha_{i} = 0$ for all $i\neq 1$. Thus, the solution $\alpha$ is equivalently encoded in the vector field $V$ such that 
\begin{equation}
\label{Vdef}
V^\flat = \alpha_1 =\alpha \in \mathfrak{X}(\mathcal{U}_p).\end{equation} 
We decompose $V=v(T-N)$ as in the splitting  \eqref{split} and may assume, after further restricting $\mathcal{U}_p$ if necessary,  that the projections of $V$ onto both summands of $T\bM = \mathbb{R}\partial_t \oplus \M$ are nontrivial as this holds for the initial data.

Next we extend the symmetric bilinear form  $\Z \in \Gamma(\mathcal{U}_p, T^*\mathcal{M}_p \otimes T^*\mathcal{M}_p)$ uniquely to a section   $\Z \in \Gamma(\mathcal{U}_p, T^*\mathcal{U}_p \otimes T^*\mathcal{U}_p)$, by demanding that $V$ inserts trivially into $\Z$. For this extended $\Z$ the evolution equation \eqref{ze} which was used to define $\Z$ then becomes equivalent to \eqref{evc} as follows from combining \eqref{zur1} and \eqref{ze}. In summary, we have constructed $(\bg,V,\Z)$  on $\mathcal{U}_p$ which satisfy the equations

\begin{eqnarray}
\bRic &=& Z - \text{Sym}(\bnab E),  \label{eq1a}\\
(d+\delta^{\bg}) V^{\flat} &=& 0, \label{eq2} \\
(\bnab_V Z)(A,B) &=&0\text{ for all  } A,B \in T\M, \label{eq3}, \\
\Z(V,\cdot) &= &0. \label{eq4}
\end{eqnarray}

On $\mathcal{U}_p$ we fix from now on a local $\bg$-pseudo-ONB $s=(s_0,....,s_n)$ as constructed in \eqref{rea}. That means, $T=s_0 = \frac{1}{\sqrt{-\bg(\partial_t,\partial_t)}}\partial_t$ is a unit timelike vector field on $\mathcal{U}_p$ which restricts on $\M_p$ to $\frac{1}{\lambda|_{\M}}\partial_t$, the unit normal vector field to $\M_p$ with respect to $\g$. Moreover, as $h(\partial_t,X) = 0$ for $X \in T\M$, it follows that the $(s_1,...,s_n)$ restricted to $\M$ are tangent to $\M$ and form a pointwise ONB for $(T\M_p,\g)$. 

\bigskip

In the subsequent calculations we simplify and abbreviate our notation for some --- otherwise very lengthy --- formulas as follows: writing
\[ A \equiv B \mod (\ldots ), \] where $A,B$ are tensor fields of the same type over $\bM$ indicates that $A=B$ up to the addition of terms which are \textit{linear} in the quantities specified in the bracket (or contractions of these quantities). The explicit formulas for these linear terms are straightforward to compute in each case but turn out to be irrelevant for our purposes. By $\bnab$ we also denote the covariant deriviatve on tensor fields induced by the Levi Civita connection of $\bg$. It follows from linearity and the product rule for $\bnab$ that
\[ A\equiv  0 \mod (C) \quad\text{implies} \quad \bnab A = 0 \mod (C,\bnab C). \]
As an example, equations  \eqref{eq3} and \eqref{eq4} imply that
\begin{align}
\bnab_V \Z \equiv\ 0 \mod(\bnab V). \label{bz0}
\end{align}
Indeed, the non vanishing terms of $(\bnab_V \Z)$ are $(\bnab_V \Z)(T,X)$
for $X$ a vector field on $\mathcal{U}_p$ which is tangent to $\M$  and $(\bnab_V \Z)(T,T) $. Both can be expressed in terms of $\bnab V$ using $V=\frac{1}{u}(T-N)$ and  equations \eqref{eq3} and \eqref{eq4}:
\begin{eqnarray*}
(\bnab_V \Z)(T,X) 
&=& (\bnab_V \Z)(N,X) - \tfrac{1}{u}\Z(\bnab_V V,X)  \ 
=\  - \tfrac{1}{u}\Z(\bnab_V V,X), \\
(\bnab_V \Z)(T,T) 
&= & (\bnab_V \Z)(N,N) - \tfrac{2}{u}\, \Z(\bnab_V V,N) 
\ =\  { -\tfrac{2}{u}\, \Z(\bnab_V V,N)}. 
\end{eqnarray*}

\subsection{PDEs for $\bnab V$ and $E$}
In the terminology of the previous subsection, we next show that the data $\bnab V$ and $E$ vanish on $\cal U=\mathcal{U}_p$ by showing that they solve a {linear} PDE for which uniqness of solutions is guaranteed. All calculations and operators are with respect to the metric $\bg = \bg^{\mathcal{U}_p}$ on $\mathcal{U}_p$ as just specified. 

We denote with 
$\Delta = \bnab^2$
the Bochner Laplacian (or connection Laplacian) for $\bg$ acting on tensors, as $\Delta B_{b\ldots c}=\bnab^a\bnab_aB_{b\ldots c}$,  in particular on $1$-forms or vector fields. When acting on $1$-forms,  it is related to the Hodge Laplacian $\Delta^{HL}$ on 1-forms via the Weitzenb{\"o}ck formula
\begin{align}\label{HLBL}
\Delta^{HL} = \Delta^{\bnab}  +\bRic,
\end{align}
where depending on the situation we consider $\bRic$ as $(2,0)$ or $(1,1)$ tensor. 

Now we aim for a second order equation for $\bnab V$. 
For this we will prove a series of Lemmas. {\em The general assumption in these  lemmas is that the system of equations \eqref{eq1a}-\eqref{eq4} is satisfied}. 
For brevity  in the proofs we will now use indices $a,b,c, \ldots $ as {\em abstract indices}, i.e., only indicating the valence of  a tensor.  The Bochner Laplacian applied to a vector field $X$ is denoted by
$\Delta X^a=\bnab^b\bnab_bX^a$, where we use Einstein's summation convention. The identity \eqref{HLBL}, for example, reads as
\[\Delta^{HL}X_a = \bnab^b\bnab_bX_a +\bRic_a^{~b}X_b.\]
We will also use expressions such as $\bnab E(V)$ or $\bnab\bnab_V E$. These are meant to be as $V$ inserted into the tensor $\bnab E\in  \otimes^2T^*\cal U$ and into $\bnab\,\bnab E\in   \otimes^3T^*\cal U$ in the respective slot, e.g., 
$\bnab E(V)\in \otimes T^*\cal U $ and 
$\bnab\, \bnab_V E\in   \otimes^2T^*\cal U$. Expressed with indices, this would be $V^a \bnab_bE_a$ and $V^a\bnab_c \bnab_bE_a$, respectively.
We also use $\Delta$ as acting on arbitrary tensors.

\blem\label{1stelem}
The tensor $\bnab V\in T^*\cal U\otimes T\cal U$ satisfies
\begin{equation} \label{1ste}
\Delta \bnab V  \equiv 0  \mod(\bnab V,   (\bnab  \bnab E)(V),  \bnab\, \bnab_V E,  \bnab_V \bnab E).\end{equation}
\elem
\bprf Using abstract index notation and 
successively interchanging covariant derivatives using the curvature tensor we obtain
\begin{eqnarray*}
\Delta \bnab_aV^b
&=&
\bnab^c\bnab_a\bnab_cV^b +V^d \bnab^c\bR_{ca\ d}^{\ \ b} +\bR_{ca\ d}^{\ \ b}\bnab^cV^d
\\
&=&\bnab_a \Delta V^b- \bRic_a^{~c}\bnab_cV^b + \bR^{d~~b}_{~a~~c}\bnab_dV^c
+ V^c\bnab^d\bR_{da~~c}^{\ \ \, b}
+\bR_{ca\ d}^{\ \ b}\bnab^cV^d
\\
&\equiv & \bnab_a \Delta V^b
+V^c \bnab^d\bR_{da~~c}^{\ \ \, b} \mod  \bnab_c V^d.
\end{eqnarray*}
To deal with the first remaining term we use equation \eqref{eq1a}, \eqref{eq4},
equation \eqref{eq2} and its consequence $0=\Delta^{HL}V^b=\Delta V^b  +\bRic^b_{~c}V^c$:
\begin{eqnarray*}
 \bnab_a \Delta V^b
 &
 =
 &
 -\bnab_a\Ric^b_{~c}V^c
 \\
 &=&
 -\bnab_a\Z^b_{~c}V^c+ \tfrac{1}{2}V^c ( \bnab_a \bnab^b E_{c}+ \bnab_a \bnab_cE^b)
 \\
& 
 \equiv &
 0\mod(\bnab_cV^d, V^e  \bnab_c \bnab^d E_{e}, V^e \bnab_c \bnab_eE^d).
\end{eqnarray*}
Finally we use the symmetries of $\R_{abcd}$ to deal with the term $ \bnab^d\bR_{da~~c}^{\ \ \, b}V^c$: 
\begin{eqnarray*}
V^c \bnab^d\bR_{da~~c}^{\ \ \, b}
 &=&
V^c  \bnab^d\bR_{~cda}^{b}
\\
&=&
-V^c \left( \bnab^b \bR_{c\ da}^{~d} +  \bnab_c \bR_{\ \ da}^{db}\right)
\\
&=&
+V^c \left( \bnab^b \bRic_{ca}-  \bnab_c \bRic_{a}^{~b}\right)
\\
&=&
V^c \left( \bnab^b \Z_{ca}-  \bnab_c \Z_{a}^{~b}\right)
+\tfrac{1}{2} V^c \left( \bnab^b \bnab_{c}E_a + \bnab^b\bnab_{a}E_c-  \bnab_c \bnab_{a}E^{b} - \bnab_c \bnab^{b}E_{a} 
\right)
\\
&\equiv &0 \mod(\bnab_cV^d, V^e  \bnab_c \bnab^d E_{e}, V^e \bnab_c \bnab_eE^d, V^e \bnab_e \bnab_cE^d),
\end{eqnarray*}
because of equations \eqref{eq1a}, \eqref{eq3}, \eqref{eq4}, and \eqref{bz0}. This verifies the lemma.
\eprf
The idea is now to prolong equation \eqref{1ste}, i.e., to derive linear equations for the $E$-dependent quantities in the brackets, which should all vanish, until we obtain a closed linear PDE system.
We start with deriving a second order equation for $E$. To this end, we introduce the ``Einstein tensor'' of $\Z$, i.e.,
\[ L:= \Z - \tfrac{1}{2}\mathrm{tr}_{\bg}( \Z) \, \bg. \]
Let $\overline{G} = \bRic- \frac{\bscal}{2}   \bg$ denote the Einstein tensor of $(\mathcal{U}_p,\bg)$. 
By equation \eqref{eq1a} we get for $X, Y \in T\mathcal{U}_p$
\begin{equation}
\begin{aligned} \label{treu}
\overline{G} =  -\text{Sym}(\bnab E)+\Z + \tfrac{1}{2}(\mathrm{tr}_{\bg}(\bnab E) - \mathrm{tr}_{\bg} Z) \cdot \bg
\end{aligned}
\end{equation}
This equation implies
\blem\label{nablaelem}
The $1$-form $E$ satisfies
\begin{align}
0 = \Delta E - \bRic(E^\sharp,.)- 2 \delta^{\bg} L. \label{nablae}
\end{align}
\elem
\bprf 
Taking the divergence $\delta^{\bg}$ on both sides of \eqref{treu} yields
\begin{eqnarray*}
0&=& \delta L_b -\tfrac{1}{2}\bnab^a\bnab_aE_b -\tfrac{1}{2}\bnab^a\bnab_bE_a+ \tfrac{1}{2}\bnab_b\bnab^aE_a
\\
&=&\delta L_b -\tfrac{1}{2}\Delta E_b+ \tfrac{1}{2} \ \bR_{b\ a}^{~a\ c}E_c
\\
&=&\delta L_b -\tfrac{1}{2}\Delta E_b+ \tfrac{1}{2} \ \bRic_{b}^{\ c}E_c,
\end{eqnarray*}
which proves the statement.
\eprf

 Next, we investigate the quantity $\bnab_V E$ and prove 
 \blem \label{x3lem}
 The $1$-form $\bnab_V E$ stisfies
  \begin{align}
\Delta(\bnab_{V} E) = \bnab_{V}  (\delta^{\bg} L) \mod(E,\bnab E, \bnab V). \label{x3}
\end{align}
\elem
\bprf Again we compute in abstract indices commuting covariant derivatives
\begin{eqnarray*}
\bnab^c\bnab_c(V^d\bnab_dE^a)
&\equiv &
V^d\bnab^c\bnab_c\bnab_dE^a\mod( \bnab_cV^d,\bnab_cE^d)
\\
&\equiv &
V^d\bnab_d\Delta E^a \mod( \bnab_cV^d,E^d ,\bnab_cE^d)
\\
&\equiv &
V^d\bnab_d(\delta L)^a \mod( \bnab_cV^d,E^d, \bnab_cE^d)
\end{eqnarray*}
by equation \eqref{nablae}.
\eprf

Next, we find a second order equation for $\bg(\bnab E,V)$, i.e., for $V^b\bnab_aE_b$.
\blem
 \label{x2lem}
 The $1$-form $(\bnab E)(V)$ satifies
 \begin{align} 
\Delta (\bnab E(V))= 0 \mod (\bnab V , E, \bnab  E )
. \label{x2}
\end{align}
\elem
\bprf Similarly as before we compute
\begin{eqnarray*}
\Delta (V^b\bnab_aE_b)&\equiv &
V^b\bnab^c\bnab_c\bnab_a E_b
\mod (\bnab_cV^d, \bnab_c E^d)
\\
&\equiv &
V^b\bnab_a \Delta E_b
\mod (\bnab_cV^d, \bnab_c E^d)
\\
&\equiv &
V^b\bnab_a (\delta L)_b
\mod (\bnab_cV^d, E^d, \bnab_c E^d)
\end{eqnarray*}
Since we computing $\mod(\bnab V)$ it is enough to compute
\[
V^b (\delta L)_b
=
V^b\bnab^c\Z_{cb}- \tfrac{1}{2}V^b\bnab_b(\mathrm{tr}(\Z))
\equiv 0\mod (\bnab V),\]
using the definition of $L$ and $V^bZ_{ba}=0$.
\eprf

Finally we derive an equation for the 1-form $\delta^{\bg} L$. 
\blem
 \label{x1lem}
The $1$-form $\bnab_V\delta L$ satisfies
\begin{equation} 
\bnab_V(\delta L)\equiv 0
\mod(\bnab V,\bnab\,\bnab V, \bnab E ).
\label{x1}
\end{equation}
\elem
\bprf Using the definition of $L$ and of $\delta^\bg$ we compute
\begin{eqnarray*}
V^b\bnab_b(\delta L)_a
&=&
V^b\left( \bnab_b\bnab^cZ_{ca} -\tfrac{1}{2}\bnab_b\bnab_a(Z_c^{~c})\right)
\\
&\equiv& 
V^b \bnab_b\bnab^cZ_{ca}\mod(\bnab V)
\\
&\equiv& 
V^b \bnab^c\bnab_b Z_{ca}
+
V^b \bR_{b~c}^{~c\ d}Z_{da}
+
V^b \bR_{b~a}^{~c\ \ d}Z_{cd}
\mod(\bnab V)
\\
&\equiv& 
V^b \bRic_{b}^{\ d}Z_{da}
+
V^b \bR_{b~a}^{~c\ \ d}Z_{cd}
\mod(\bnab V)
\\
&\equiv& 
V^b\, \bR_{a~b}^{~d\ \ c}Z_{cd}
\mod(\bnab V,\bnab E ),
\end{eqnarray*}
because of equation \eqref{eq1a}. The term $V^b\, \bR_{a~b}^{~d\ \ c}$ however is a linear expression in the second and first covariant derivatives of $V$, and hence the claim follows.
\eprf
%
%
\subsection{Reformulation of the PDEs in terms of differential operators}
\label{wavesubsec}
Now we want to use the PDEs derived in the previous section as a ``wave equation'', i.e., in terms of a differential operator involving $\Delta$.
We introduce the following vector bundle over $\mathcal{U}_p$:
\begin{eqnarray*}
{\mathcal{E}}
&:= &
(T^*\M\otimes T\M )\+ T^*\M \+ T^*\M\+ T^*\M
\end{eqnarray*}
The vector bundle ${\mathcal E}$ carries a covariant derivative naturally induced by $\bnab$ and denoted by the same symbol. Moreover, there is an operator $\Delta$ of Laplace type on ${\mathcal E}$ which is given by taking the Bochner-Laplacian $\Delta$ in each summand and letting this operator act diagonally on sections, i.e.,
 \begin{align*} 
\Delta = \begin{pmatrix} \Delta & & \\ & \ddots & \\ & & \Delta
\end{pmatrix}.
 \end{align*}
 Using the solutions of the previous section we define the sections $\eta\in \Gamma (\cal E|_{\cal U})$ and $\xi\in \Gamma (T^*\cal U)$ by
 \begin{equation}
 \label{zd4}
\eta := \left(\bnab V,E,\bnab_V E,  (\bnab E)(V)\right), \qquad
\xi:= \delta^{\bg} L
 \end{equation}
 Combining the equations in Lemmas \ref{1stelem}, \ref{nablaelem}, \ref{x3lem}, \ref{x2lem} and \ref{x1lem} we obtain
 \begin{Proposition}
The sections  $\eta$ and $ \xi$ as defined in \eqref{zd4} solve the coupled \textit{linear} PDE
 \begin{eqnarray}
 \Delta\eta &= & F(\eta,\bnab\eta,\xi),  \label{fauch} \\
 \bnab_V\xi &=&H(\eta,\bnab \eta) , \label{fauch2}
 \end{eqnarray}
 where $F$ and $H$  are certain sections of ${\mathcal E}|_{\mathcal U}$ and $T^*\cal U $ which depend {\em linearly} on the indicated quantities. 
 \end{Proposition}
Now suppose that $\eta$ and $\xi$ are arbitrary sections of ${\mathcal E}|_{\mathcal U}$ and $T^*\cal U$ and interpret the left hand side of \eqref{fauch} as a linear differential operator acting on these sections. Moreover, we trivialise the bundles ${\mathcal E}|_{\mathcal U}$ and $T^*\cal U$ with respect to the fixed coordinates $(x^0,...,x^n)$ on $\mathcal{U}$ and view in terms of this identification $\eta \in C^{\infty}(\mathcal{U},\mathbb{R}^{N})$, where $N =n^2+3n$, and $\xi\in C^\infty(\cal U,\rr^{n+1})$.

 \begin{Proposition}
 In the fixed local trivialisation, equations \eqref{fauch} and \eqref{fauch2} imply a \textit{linear} symmetric hyperbolic first order PDE
 \begin{align}
 A^0(t,x,\eta,\partial\eta,\xi) \partial_0 \begin{pmatrix}\eta \\ \partial\eta \\\xi\end{pmatrix} = A^i(t,x,\eta,\partial\eta,\xi) \partial_i \begin{pmatrix}\eta \\ \partial\eta \\\xi\end{pmatrix} + b(t,x,\eta,\partial\eta,\xi) \label{zd3}
 \end{align}
for $\eta$ and $\xi$, i.e., $b$ depends linearly on $(\eta,\partial\eta,\xi)$.
 \end{Proposition}
 
 \bprf
 The proof uses only that the linear second order operator
 \begin{align}
 P:=\Delta -F (\cdot,\cdot,\xi) \label{n1}
 \end{align} acting on $\mathcal{E}|_{\mathcal U}$ is  \textit{normally hyperbolic} for each $\xi$. 
 In general, given any tensor bundle $\mathcal{E} \rightarrow \mathcal{U}$ trivialised by the coordinates $x^i$, i.e. $\mathcal{E} \cong \mathcal{U}_p \times \mathbb{R}^N$,  and any linear second order differential operator $P: \Gamma(\mathcal{E}) \rightarrow \Gamma(\mathcal{E})$, we say that $P$ is normally hyperbolic if its principal symbol is given by the metric, i.e. in the local trivialisation
\begin{align*} P = -  \bg^{\mu \nu}(p) \frac{\partial^2}{\partial_{\mu} \partial_{\nu}} + M^{\mu}(p) \frac{\partial}{\partial x^{\mu}} + K(p) 
\end{align*}
for matrix-valued coefficients $M_{\mu}$ and $K$ depending smoothly on $p$. Note that in our case  the term $F$ in \eqref{n1} only affects the matrices $M$ and $K$ but not the symbol. If $\eta=(\eta_1,...,\eta_N)   \in C^{\infty}({\mathcal{U}_p},\mathbb{R}^N)$ is arbitrary, the equation $P \eta = 0$ can be rewritten as linear first order equation by applying formally the same steps as before when \eqref{eq11} was rewritten as first order equation: For $A=0,...,N$ we introduce the
 quantities $k_{A}:= \partial_t \eta_{A}$ and $\eta_{A, i}:= \partial_{i} \eta_{A}$. In terms of these quantities, $P\eta=0$ implies that
\begin{align}
\partial_t \eta_{A} &= k_{A}, \label{ff1} \\
\bg^{ij} \partial_t \eta_{A, i} &= \bg^{ij} \partial_{i} k_{A}, \label{ff2} \\
- \bg^{00} \partial_t k_{A} &= 2 \bg^{0 j} \partial_{j} k_{A} + \bg^{ij} \partial_{j} \eta_{A, i} +H_0^{A B} k_{B}+  \sum_{i=1}^{n} H^{A B}_{i} \eta_{B, i} + K^{A B}\eta_{B} , \label{ff3}
\end{align}
holds\footnote{This system is even equivalent to $P\eta=0$. Indeed, let a triple $(\eta_{A},k_{A},\eta_{A,i})$ solve \eqref{ff1}-\eqref{ff3}. As $g^{ij}$ is invertible for $g$ sufficiently close to $h$, \eqref{ff2} is the same as $\partial_t \eta_{A,i}=\partial_{i}k_{A}$, and \eqref{ff1} then gives $\partial_t (\eta_{A,i}-\partial_{i} \eta_{A})=0$. Appropriate choice of initial data  ensures $\eta_{A,i}=\partial_{i}\eta_{A}$ at $t=0$ and thus equality everywhere. Then \eqref{ff3} is nothing but $P \eta = 0$.}.
Equations \eqref{ff1}-\eqref{ff3} applied to our operator \eqref{n1} and sections $\zeta_i$ yield
 \begin{align} 
 A_1^0(t,x) \partial_0 \begin{pmatrix}\eta \\ \partial\eta \end{pmatrix} =  A_1^i(t,x) \partial_i \begin{pmatrix}\eta \\ \partial\eta \end{pmatrix} + b_1(t,x,\eta,\partial\eta,\xi). \label{zd2}
 \end{align}
It is easy to read off an explicit form of the matrices $A^{\mu}_1$ and to see that they are symmetric and that $A^0_1$ is positive definite as $\bg \in \mathcal{G}_p$. 

We turn to equation \eqref{fauch2}.
We write $\xi = \xi^{\mu} \partial_{\mu}$ and $V = u_t(T-N_t) = u_t(\frac{1}{\sqrt{-\bg(\partial_t,\partial_t)}}\partial_t - N_t^i \partial_i)$. In terms of these quantities, equation  \eqref{fauch2} is equivalent to
\begin{align} \label{zd1}
 \partial_t (\xi^{\mu})_{\mu=0,...,n} = \sqrt{-\bg(\partial_t,\partial_t)} N_t^i   \partial_i (\xi^{\mu})_{\mu = 0,...,n} + b_2(t,x,\eta,\partial\eta,\xi)
\end{align}
where  $b_2$ depends linearly on $(\eta,\partial\eta)$ via $H$ and linearly on $\xi$ via  contractions of $\xi$ with Christoffel symbols for $\bg$ which results from writing $\nabla = \partial + \Gamma$. Combining \eqref{zd2} and \eqref{zd1} gives \eqref{zd3}.
 \eprf

\subsection{Initial data and the vanishing of $\bnab V$ and $E$}
In this section we will show that $E$ and $\bnab V$ vanish everywhere on $\cal U$ and that $V$ is a null vector field.
We will achieve this by showing that 
the data $\eta$, $\xi$ and $\bnab\eta$  as defined in \eqref{zd4}, and containing the tensors $\bnab V$ and $E$,
 vanish on $\cal U$. 
 The data $\eta$, $\xi$ and $\bnab\eta$  were 
solutions of the  linear system \eqref{fauch} and \eqref{fauch2}. Hence, using the uniqueness result for solutions to \eqref{zd3}, it suffices to show that  $\eta$, $\xi$ and $\bnab\eta$ vanish on $\M_p$ (which for simplicity we will denote by $\M$ in the following) in order to obtain  that $\bnab V$ and that $E=0$. Moreover we show that $V$ is null on $\M$ which will imply, by $V$ being parallel, that  it  $V$ is null everywhere.

\begin{Proposition}
The vector field $V$ defined in equation \eqref{Vdef} and the sections defined in  equation \eqref{zd4} of Section \ref{wavesubsec} satisfy equations along the initial hypersurface $\M$,
\begin{equation}
 \label{nulls}
\bg(V,V)|_\M=0,\qquad \eta|_{\M} = 0, \qquad
\bnab_T\eta|_{\M} = 0, \qquad
\xi|_{\M} = 0.
\end{equation}
In particular, $\bnab V$  and $E$ vanish on $\M$.
\end{Proposition}
\bprf
In this proof all equations are understood as being evaluated on $\M$, more precisely on $\M_p = \M \cap \mathcal{U}_p$ only, i.e. we do not always write the restriction $|_\M$ after each expression here. Recall that $\eta$ and $\xi $ were defined as
\[
\eta = \left(\bnab V,E,\bnab_V E,  (\bnab E)(V)\right), \qquad
\xi= \delta^{\bg} L,\quad\text{ with }L= \Z - \tfrac{1}{2}\mathrm{tr}_{\bg}( \Z) \, \bg
\]
In the following, the order in which the initial conditions are verified turns out to be very important. First note that along $\M$ we have that $\bg=h$, where $h$ is the background metric, and hence that $T\M$ is orthogonal to $T$.  Moreover, we will not distinguish between $E$ and $E^\sharp$.
It follows from the identity \eqref{ij1} and initial data for $V$, i.e. from
\begin{align}
V|_{\M} = uT - U \label{lich},
\end{align} that the imposed constraint equation \eqref{cons} is equivalent to $\mathrm{pr}_{T\M}(\bnab_X V)|_{\M} = 0$, i.e., to
\[\bg(\bnab_X V,Y)|_\M=0,\qquad\text{ for $X,Y\in T\M$.}\]  Moreover, equation \eqref{lich} implies that
\begin{align}
\bg(V,V)|_{\M} = 0. \label{tr}
\end{align}
Differentiating this in direction of $X \in T\M$ yields
\[
0 =  \bg(\bnab_X V, V)  
=  u\bg(\bnab_XV, T),
\]
from which follows that $\bnab_X V = 0$ on $\M$ for $X \in T\M$. The evolution equation $(d+\delta) V^{\flat} = 0$ reduces then on $\M$ to 
\begin{align*}
c(T) \circ \bnab_{T} V^{\flat} = 0.
\end{align*}
Multiplying this from the left with $c(T)$ yields that $\bnab_T V =0$ on $\M$ and hence that 
\begin{equation}\label{nabvm}\bnab V|_\M=0.\end{equation} Moreover, the initial data for $\bg$ were chosen in Subsection \ref{idat} precisely in such a way that 
\begin{equation}\label{em}
E|_\M=0.\end{equation}
This also implies 
that  \begin{equation}\label{nabetm}
\bnab_{X} E|_{\M}=0,\text{ for $X\in T\M$}.
\end{equation}
Showing that the remaining quantities in $\eta,$ $\bnab\eta$ and $\xi$ vanish along $\M$ is rather involved. Again for brevity, we will use abstract index notation with indices $a,b,c, \ldots $ ranking from $0$ to $n$. We will however abuse this abstract index notation as indicated before,  when writing a $0$ for a contraction $B(T, \ldots) $ of a tensor $B$ with the vector field $T$,
\[T^aB_{ab\ldots}=B_{0b\ldots},\]
but also when using indices $i,j,k,\ldots$ ranging from $1$ to $n$ and referring to directions in $T\M$. Since  along $\M$ the vector field $T$ is orthogonal to $T\M$ we have that $\bg_{0i}=h_{0i}=0$, as well as $\bg_{00}=-1$ and $\bg_{ij}=\g_{ij}$.

We will start by showing that the  initial data specified for $\Z$ imply that  $\bnab_T E$ vanishes on $\M$, i.e., that $\bnab_0 E_a=0$ along $\M$.
Starting point is equation \eqref{treu}, which in indices reads as
\begin{equation}\label{treu-ind}
\overline{G}_{ab}=-\bnab_{(a}E_{b)}+Z_{ab}+\tfrac{1}{2}(\bnab_cE^c-Z_{c}^{~c}) \bg_{ab}
.\end{equation}
Evaluation on $\M$ using the hypersurface formula \eqref{hsf},
\[
\overline{G}_{0i} =
\nabla_k \W^k_{~i}+\bnab_i\W_k^{~k},\]
implies that 
\begin{equation} \label{exp}
\tfrac{1}{2}\bnab_0E_i -\tfrac{1}{2}\bnab_iE_0
\ =\ 
-\nabla_k \W^k_{~i}+\bnab_i\W_k^{~k}-Z_{0i} 
\ =\ 
-\nabla_k \W^k_{~i}+\bnab_i\W_k^{~k}-\tfrac{1}{u}U^k\Z_{ki}
\ =\ 0,
\end{equation}
because of 
$0=\Z(V,\cdot) = u\Z(T,X) -\Z(U,X)$ and  the first  initial condition in \eqref{iv} for $\Z$. But now $E_a$ is zero along $\M$ and hence is $\bnab_iE_0$, and so we obtain that  
\begin{equation}
\nabla_{0}E_i = 0. \label{67}
\end{equation}
Hence, it remains to prove that also 
$\bnab_0E_0
{=} 0$.

To this end, recall the hypersurface formula \eqref{hsf} contracted with $T$ twice,
\[\overline{G}_{00}=\tfrac{1}{2}(\scal^\g - \W_{ij}\W^{ij}+(W_i^{~i})^2),\]
and the above formula \eqref{treu-ind} to obtain 
\[
\begin{array}{rcl}
\tfrac{1}{2}(\scal^\g - \W_{ij}\W^{ij}+(W_i^{~i})^2)
&=& - \bnab_0E_0 + \Z_{00}
-
\tfrac{1}{2}\bnab_c E^c +\tfrac{1}{2}\Z_c^{~c}
\\[2mm]
&=&
-\tfrac{1}{2} \bnab_0E_0 +\tfrac{1}{2} \Z_{00}
-
\tfrac{1}{2}\bnab_k E^k +\tfrac{1}{2}\Z_k^{~k}.
\end{array}\]
Hence, using again $\Z(V,.)=0$ and $\bnab_{i}E_j=0$ along $\M$, we get
\begin{equation}
\label{78}
\tfrac{1}{2} \bnab_0E_0\ =\ 
\tfrac{1}{2}N^i N^j \Z_{ij}
+\tfrac{1}{2}\Z_k^{~k}
-\tfrac{1}{2}(\scal^\g - \W_{ij}\W^{ij}+(W_i^{~i})^2).
\end{equation}
The next lemma shows that this term vanishes:
\begin{Lemma}
On $\M$ satisfying the constraint \eqref{cons}  it holds that
\begin{equation}
\scal^\g - \W_{ij}\W^{ij}+(W_k^{~k})^2
\ =\ 
 N^i N^j \Z_{ij}
+\Z_k^{~k}.
\label{grz}
\end{equation}
%
%
%
\end{Lemma}
\bprf
On $\M$ we have $N = \frac{1}{u}U$ with $u^2=\bg(U,U)$. 
An easy consequence this and of the constraint \eqref{cons}, i.e., of  $\nabla_iU_j+u\W_{ij}=0$, is the formula
\begin{align}
\nabla_i u =  -u N^k\W_{ik}, \label{norm}
\end{align} 
and the resulting
\begin{equation}\label{nabN}
\nabla_iN_j=
N^k\W_{ki}N_j-\W_{ij}.
\end{equation}
Now we use both initial conditions \eqref{iv} for $\Z$ to first determine 
\[
\Z_k^{~k}
=
 \mathrm{scal}^{\g} -\W_{ij}\W^{ij}  +  (\W_k^{~k})^2 +
 N^iN^j\left(
2 \W_{ki}\W^{k}_{~j}  -2\Ric_{ij} -2\W_{k}^{~k}  \W_{ij} 
+\Z_{ij}\right),
%
%
\]
and then compute, using  the constraint \eqref{cons} and equations \eqref{norm} and \eqref{nabN}, that
\begin{eqnarray*}
N^iN^j\Z_{ij}
&=&
N^i\nabla_i\W_j^{~j}
-N^i\nabla^j\W_{ij}+
\\
&=&
\tfrac{1}{u^2} N^i\nabla_iu \nabla_jU^j -\tfrac{1}{u} N^i\nabla_i\nabla_j U^j
-\tfrac{1}{u^2} N^i\nabla_ju \nabla_iU^j +\tfrac{1}{u} N^i\nabla_j\nabla_i U^j
\\
&=&
N^iN^k\W_{ik}\W_j^{~j} - N^k\W_{kj}\W_{i}^{~j}N^i
+N^iN^k\R_{jik}^{\ \ \ j}
\\
&=&
- N^iN^j\left( \W_{kj}\W_{j}^{~k}-
\W_{ij}\W_k^{~k}-\Ric_{ik}\right).
\end{eqnarray*}
Putting these two equations together gives the  desired equation \eqref{grz}.
\eprf
Hence we have established that $\bnab_0E_0=0$.
Combined with \eqref{67} this yields $
\bnab_{0}E |_\M= 0$.
It then follows automatically that $V^a\bnab_b E_a= V^a\bnab_a E_b = 0$ on $\M$ . Altogether we have now that 
\[\bnab_a E|_\M=0.\]
Then equation \eqref{eq1a} yields as immediate consequence  that
\begin{align}
V^a \bRic_{ab}|_{\M} = 0.
\end{align}
We turn to $\bnab_0 \bnab_i V^a$-terms for $X \in T\M$ and want to show that such expressions vanish on $\M$. As $\bnab V = 0$ on $\M$ we have that
\begin{align}\label{vanish}
\bnab_0 \bnab_i V^b = \bR_{0ia}^{\ \ \ b} V^a
\end{align}
on $\M$, and we have to show that this expression vanishes. Note that as a further consequence of $\bnab_a V = 0$ on $\M$ we have
\begin{align}\label{rvmm}
\bR_{ijb}^{\ \ \ a}V^b = 0
\end{align}
on $\M$. We will now prove that   $V^a\bR_{0abc} = 0$ on $\M$ for all $b,c=0, \ldots, n$:

First we note that 
\[
V^a\bR_{0aij} =- V^a\bR_{ija0}=0,\]
because of \eqref{rvmm}.
Next we use $V^a\bRic_{ab}=0$ to get
\[  V^a\bR_{0ai}^{\ \ \ 0}=-V^a\bRic_{ab}+V^a\bR_{jai}^{\ \ \ j}=0,\]
again because of \eqref{rvmm}. This implies $V^a\bR_{0abc} = 0$ on $\M$.

Hence, the vanishing or the term \eqref{vanish} 
 is equivalent to
\begin{align}
0=V^a\bR_{ia0j}
\label{soll}
\end{align}
for all $i,j=1,\ldots, n$.
In fact, because of  $V= u(T-N)$, it suffices to 
 prove \eqref{soll} for $X^i,Y^j \in N^{\perp} \subset T\M$, i.e., with $X^iN_i=Y^iN_i=0$. We use now indices $r,s,t=1, \ldots, n-1$ for tensors from $N^\perp$. Using this convention 
  we rewrite
\[
V^a\bR_{ra0s}=-uN^i\bR_{ri0s}+\bR_{r00s}
=
-uN^iN^j\bR_{rijs}-u\bRic_{rs}+u\bR_{ri\ s}^{\ \ i}.
\]
Because of $\bnab E=0$ and the resulting  
 $\Z = \bRic$ on $\M$, this means that 
 \eqref{soll} is verified if and only if
 \begin{align}
\Z_{rs} =
-N^iN^j\bR_{rijs}+\bR_{ri\ s}^{\ \ i}.
\label{g}
\end{align}
Now we use the Gau\ss\ equation
\begin{align*}
\bR_{ijkl}= \R_{ijkl}-\W_{i[k}\W_{l]j}
\end{align*}
in order to rewrite the curvature terms in \eqref{g} in terms of data on $\M$. The  rewritten equation \eqref{g} is precisely the defining initial condition for $\Z_{rs}$, i.e., equation \eqref{iv}. This proves \eqref{soll} and thus $\bnab_0 \bnab_i V^b = 0$. As $V^a\bRic_{ab} = 0$ on $\M$, we have that 
\[0 = \Delta^{HL}V^b= \Delta V^b
=\bnab_a\bnab^aV^b
=-\bnab_0\bnab_0V^b=0,
\]
because of $\bnab_i\bnab_jV^a=0$ on $\M$.

\bigskip

The last and most complicated part of the proof now consists of showing that $\delta^{\bg} L_a$ and $V^a \bnab_0 \bnab_a E_b$ vanish on $\M$. As a starting point, we take equation \eqref{eq1a}, 
\[\bRic_{ab} = \Z_{ab} - \bnab_{(a} E_{b)}.\]
Differentiating both sides covariantly in direction of $V$, using that $\bnab_aV^a\Z_{bc}=0$ along $\M$ due to \eqref{bz0}, yields 
\begin{equation}
\label{vric}
V^a\bnab_a\bRic_{bc}=
V^a\bnab_a \bnab_{(b} E_{c)}
=
V^a\bR_{a(b\ c)}^{\ \ \ d}E_d+V_a\bnab_{(b}\bnab^aE_{c)}.
\end{equation}
Setting $b=0$ and $c=j$ this and $E_d|_\M=0$  and $\bnab_aE_b|_\M=0$ gives
\begin{equation}\label{fit}
V^a\bnab_a\bRic_{0j}=\tfrac{1}{2}V^a\bnab_{0}\bnab_aE_{j}
=
\tfrac{u}{2}\bnab_0 \bnab_0 E_j - N^i\bnab_i \bnab_0 E_j = \tfrac{u}{2}\bnab_0 \bnab_0 E_j
.\end{equation}
%
We now show that $V^a\bnab_a\bRic_{0j}$ vanishes on $\M$: 

Using the second Bianchi identity we find
\begin{equation}\label{cod}
V^a\bnab_a\bRic_{0j}
=
V^a\bnab_a\bR_{b0j}^{\ \ \  \ b}
=
-V^a\bnab_a\bR_{ji0}^{\ \ \ \ i}
=
-
V^a\bnab_j\bR_{ia0}^{\ \ \ \ i}
-
V^a\bnab_i\bR_{aj0}^{\ \ \ \ i}=0,
\end{equation}
because of equation \eqref{soll} along $\M$ and differentiation is along $\M$.
%
Thus, \eqref{fit} gives
\begin{align}
\bnab_0 \bnab_0 E_j=0.
\label{fitt}
\end{align}
This as well as 
 $E_a=0$ and $\bnab_aE_b = 0$ on $\M$ that also 
 \[ \Delta E_i = -\bnab_a\bnab^a E_i=0
  \]  on $\M$. But then, equation \eqref{nablae} immediately yields
\begin{align*}
(\delta^{\bg} L)_i = 0
\end{align*}
for all $j=1, \ldots , n$.
Moreover, Lemma \ref{x1lem} says that $V^a(\delta^{\bg}L)_a$ is a linear expression in $\bnab_b V^d$ and $\bnab_b\bnab_c V^d$ terms, which vanish on $\M$. It follows that on $\M$
\begin{align*}
(\delta^{\bg}L)_0= \frac{1}{u}V^a(\delta^{\bg}L)_a + N^i(\delta^{\bg}L)_i = 0,
\end{align*}
and as a consequence, 
\[ \delta^{\bg} L = 0 \text{ on } \M. \]
Again using formula \eqref{nablae} shows now that $\bnab_0 \bnab_0 E_b = 0$ on $\M$, Inserting this into \eqref{cod} shows that 
\begin{align*}
V^a\bnab_0 \bnab_a E_b = 0.
\end{align*}
Hence, all covariant derivatives $\bnab_{a} \bnab_{b} E_c$ vanish, proving equations \eqref{nulls} in the proposition.
\eprf

\section{Proof of Theorem \ref{cauchy-vf-theo}: global aspects}\label{sec5}
\subsection{From local to global}
We next globalise the local development of the initial data. So far we have constructed for every $p \in \M$ the data $({\bg}^{\mathcal{U}_p},V^{\mathcal{U}_p},\Z^{\mathcal{U}_p})$ defined on some open $\mathcal{U}_p \subset \mathbb{R} \times \M$ sufficiently small. Let $p,q \in \M$ and assume that $\mathcal{U}_p \cap \mathcal{U}_q \neq \emptyset$. Choose coordinates $(x^0,...,x^n)$ and $(y^0,...,y^n)$ on $\mathcal{U}_p$ and $\mathcal{U}_q$ respectively as before. On $\mathcal{U}_p \cap \mathcal{U}_q$ we consider the coordinates given by restriction of the $x^i$. Then, with respect to these coordinates, the data
\begin{align*}
w_p = \left( {\bg}_{\mu \nu}^{\mathcal{U}_p} , \  V_{\mu}^{\mathcal{U}_p} ,\ \Z_{ij}^{\mathcal{U}_p} \right), \quad
w_q = \left( {\bg}_{\mu \nu}^{\mathcal{U}_q} ,\ V_{\mu}^{\mathcal{U}_q} ,\ \Z_{ij}^{\mathcal{U}_q} \right),
\end{align*}
 solve by construction the system \eqref{eq1a}-\eqref{eq3} formulated locally in the $x$-coordinates. This follows as these equations are manifestly coordinate invariant. Moreover the initial data $(w_p)_{|\M_p} = (w_q)_{|\M_q}$ coincide since they arise as restrictions of globally defined data on $\M$. It then follows from the uniqueness result for solutions of symmetric quasilinear hyperbolic systems that 
\begin{align} \label{upuq}
w_p = w_q \text{ in } \mathcal{U}_p \cap \mathcal{U}_q.
\end{align}
We now set 
\[\bM:= \cup_{p \in \M} \mathcal{U}_p \subset \mathbb{R} \times \M. \] 
As each $\bg^{\mathcal{U}_p}$ lies in $\mathcal{G}_p$, the $\bg^{\mathcal{U}_p}$ define a global Lorentzian metric on $\M$ on which $\partial_t$ is a timelike vector field. We equip $\M$ with the time orientation induced by $\partial_t$. By the previous local constructions, $(\M,\g)$ embeds into $(\bM,\bg)$ with Weingarten tensor $\W$. 
Moreover, by \eqref{upuq} the locally defined vector fields $V^{\mathcal{U}_p}$ give rise to a vector field $V \in \mathfrak{X}(\bM)$ which is parallel and of length zero as this holds locally. 

\subsection {$\M \subset \bM$ is a spacelike Cauchy hypersurface.} \label{rrr}
To this end, let $\gamma:I \rightarrow \bM = \cup_{p \in \M} \mathcal{U}_p$ be an inextendible timelike curve and let $t^* \in I$ be any fixed parameter. Let $p \in \M$ such that $\gamma(t^*) \in \mathcal{U}_p$. For such fixed $p$ we consider the restricted curve
\[ \gamma_{|\gamma^{-1}(\mathcal{U}_p)} : \gamma^{-1}(\mathcal{U}_p) \rightarrow \mathcal{U}_p, \]
which is an inextendible timelike curve in the globally hyperbolic manifold $(\mathcal{U}_p,g^{\mathcal{U}_p})$. Thus, the spacelike Cauchy hypersurface $\M_p \subset \M$ in $\mathcal{U}_p$ is met by $\gamma_{|\gamma^{-1}(\mathcal{U}_p)}$. It remains to show that $\gamma$ meets $\M$ at most once. With respect to  the splitting $\mathbb{R} \times \M$ we decompose
\[ \gamma = (\gamma_t, \gamma_{\M}) \]
and compute
\begin{align*}
0 > \bg(\dot \gamma_t \partial_t, \dot \gamma_t \partial_t) + 2 \cdot \bg(\dot \gamma_t \partial_t, \dot \gamma_{\M}) + \bg(\dot \gamma_{\M}, \dot \gamma_{\M}).
\end{align*}
Let us assume that there is $\tau \in I$ with $\dot \gamma_t(\tau) = 0$. Let $q:= \gamma(\tau)$. Then $0 > \bg^{\mathcal{U}_q}(\dot \gamma_{\M}(\tau), \dot \gamma_{\M}(\tau))$. This, however, contradicts the condition $\bg \in \mathcal{G}_q$ imposed on $\mathcal{U}_q$ and $\bg^{\mathcal{U}_q}$ in the second step of the proof. Consequently, $\gamma_t:I \rightarrow \mathbb{R}$ is strictly monotone, and thus $\gamma_t = 0$ has at most one solution. In total, $\gamma$ intersects $\M$ exactly once. It follows that $(\bM,\bg)$ is globally hyperbolic with Cauchy hypersurface $\M$ and parallel null vector $V$. 
\subsection{The metric is of the form $-\widetilde{\lambda}^2dt^2 + \g_t$}
Here we will prove the last aspect of Theorem~\ref{cauchy-vf-theo}, namely that the metric $\bg$ obtain in this way is of the form
\[\bg=-\widetilde{\lambda}^2dt^2 + \g_t.\]

Let $t$ denote the function $\bM \ni (t,x) \mapsto t$. By construction, the vector field $\text{grad}^{\bg}(t)$ is a global timelike vector field on $\bM$ and the leaves of  the integrable distribution $(\text{grad}^{\bg}(t))^{\perp}$ are the $t$-levels $\{t\} \times \M =: \M_t$. Let $F \in \mathfrak{X}(\bM)$ denote the vector field that is proportional to $\text{grad}^{\bg}(t)$ and such that $dt(F)\equiv 1$, i.e.,
\[F=\frac{1}{dt(\text{grad}^{\bg}(t))}\text{grad}^{\bg}(t),\]
and denote by $\phi$ its flow. Note that $\phi$ sends level sets to level sets, i.e.,
$\phi_s(p)\in \cal \M_{t(p)+s}$. Indeed, for each $p\in \bM$, the function $f(s):=t(\phi_s(p))\in \mathbb{R}$ satisfies
\[f'(s)=dt|_{\phi_s(p)} (F)\equiv 1,\]
and hence $f(s)=t(\phi_s(p))=s+t(p)$. We further define two open neighbourhoods $\bM^{1,2}$ of $\M$ in $\bM \subset \mathbb{R} \times \M$, 
\begin{align*}
\bM^1 &:= \{ (t,x) \in \bM \mid \phi_t(x) \text{ exists} \}, \\
\bM^2 &:= \{ p \in \bM \mid \exists \tau \text{ such that } \phi_{-\tau}(p) \in \M_{0} \},
\end{align*}
where we identify $x\in \M$ with $(0,x) \in \M_0$. Note that for each $p=(t,x) \in \bM^2$ the number $\tau = \tau(p)$ is uniquely determined. Namely if $\phi_{-\tau_1}(p),\phi_{-\tau_2}(p)  \in \M_0$ it follows from applying the function $t$ that 
\begin{align*}
 0 &= t(\phi_{-\tau_1}(p))-t(\phi_{-\tau_2}(p))
= t(p)-\tau_1 - (t(p)-\tau_2).
 \end{align*}
 Moreover, as $\bM = \cup_{x \in \M} \mathcal{U}_x$ and each $\mathcal{U}_x$ satisfies by construction \eqref{finalstep}, it simply follows that $\bM^2 = \bM$.
Then we have a well defined diffeomorphism
\begin{eqnarray*}\Psi: \bM^1 \ni (t,x) \mapsto& \phi_t(x) \in \bM,
\end{eqnarray*}
with  $\Psi ((t,x)) \in \M_t$.  Its inverse is given by
\[
\Psi^{-1}(p)=\big(\tau(p), \phi_{-\tau}(p)\big).
\]
It also satisfies
\[d\Psi_{(\widetilde{t},x)} (\del_t)=F|_{\phi_{\widetilde{t}}(x)}.
\]
Therefore, for the pulled back metric we have at each $(\widetilde{t},x) \in \bM^1$ and for every vector field $X$ on $\bM^1$ with values in $T\M_{\tau}$ for some fixed $\tau$
\begin{align}
(\Psi^*\bg)_{(\widetilde{t},x)}(\del_t,\del_t)&= \bg_{\phi_{\widetilde{t}}(x)}(F,F)\ =\frac{1}{\bg(\text{grad}^{\bg} t,\text{grad}^{\bg} t)}, \label{la}
\\
(\Psi^*\bg)_{(\widetilde{t},x)}(\del_t,X)&= \bg_{\phi_{\widetilde{t}}(x)}(F,d\psi_{(\widetilde{t},x)}(X))\ =\ 0,
\end{align}
since $d\psi_{(\widetilde{t},x)}(X)$ is  tangential to $\M_{\widetilde{t}+\tau}$  and $F$ is a multiple of the gradient of $t$. Hence, $\partial_t$ is orthogonal to $T\M$ with respect to $\Psi^*\bg$, showing that 
\begin{align}
\Psi^*\bg = -\widetilde{\lambda}^2dt^2+\g_t \label{zur}
\end{align}
for some $t$-dependent family of metrics on $\M$. As $\Psi$ restricts to the identity on $\M \subset \bM^{1,2}$ it follows that $\g_0 = \g$. Moreover, by equation \eqref{la},
\[ \widetilde{\lambda}{=}\frac{1}{\sqrt{\bg(\text{grad}^{\bg} t,\text{grad}^{\bg} t)}}. \] 
On $\M=\M_0$ we have $\widetilde{\lambda}|_{\M}=\frac{1}{\sqrt{h(\text{grad}^{h} t,\text{grad}^{h} t)|_{\M}}} = \lambda|_{\M}$. In summary, passing from $(\bM,\bg)$ to $(\bM^1,\psi^*\bg)$ via $\psi$ yields an open neigborhood of $\M$ in $\mathbb{R} \times \M$ with parallel null vector field, metric of the form \eqref{zur}, and as $\psi$ restricts to the identity on $\M$, we deduce that $\M$ is also a spacelike Cauchy hypersurface for $(\bM^1,\psi^*\bg)$. This finishes the proof of Theorem \ref{cauchy-vf-theo}.
$\hfill \Box$


\begin{re}
The proof of Theorem \ref{cauchy-vf-theo} shows that $\bg$ depends on the background metric $h$ which was introduced in the proof in terms of the following PDE-system: the contracted difference tensor $E$ of the Levi-Civita connections of $\bg$ and $h$ vanishes, i.e.
\begin{align}
E(X) = -\mathrm{tr}_{\bg} \left(\bg(A(\cdot,\cdot),X) \right) = 0 \text{ for all  } X \in TM,
\end{align}
where $A(Y,Z) := \bnab_Y Z - \nabla^h_Y Z$ for $Y,Z \in TM$.
Imposing this extra condition in Theorem \ref{cauchy-vf-theo} for the solution $\bg$ for a fixed background metric $h$ determines $\bg$ uniquely for each choice of $h$.
\end{re}

\section{Riemannian manifolds satisfying the constraint}
\label{constraintsec}
In this section we study Riemannian manifolds $(\M,\g)$  satisfying the constraint condition~\eqref{cons}, which in fact means that there is a non-zero vector field $U$ such that $\nabla U$ is a symmetric endomorphism of $(T\M,\g)$. 
\subsection{The local structure and the proof of Theorem \ref{pf}}
The condition  \eqref{cons} is equivalent to $\nabla U^\flat=\g(\nabla U, \cdot)$ being symmetric, which in turn is equivalent to  $dU^{\flat} = 0$. 
Now we can argue analogously as in \cite[Proposition 8]{leistner-schliebner13}:

Locally near some fixed $x_0 \in \M$ we have that $U = \text{grad}^{\g} (z)$ for some function $z$ on $\mathcal{V} \subset \M$ with $z(x_0) = 0$. The leafs of the integrable distribution $U^\bot=\ker(dz)$ are given by the level sets  \[\cal U_c=\{x\in \mathcal{V} \mid z(x)=c\}.\]  Let $Z\in \Gamma (\mathcal{V})$ denote the vector field that is proportional to $U$ and such that $\g(U,Z)=dz(Z)\equiv 1$, i.e.,
\[Z=\frac{1}{dz(\text{grad}^{\g} z)}\text{grad}^{\g} z = \frac{1}{\g(U,U)} U,\] 
and denote by $\phi$ its flow. Choose $\epsilon > 0$ and an open subset $\mathcal{W} \subset \M$ centered around $x_0$ such that $\phi$ is defined on $(-\epsilon, \epsilon) \times \mathcal{W}$.  We now restrict the levels $\cal U_c$ to their intersections with $\mathcal{W}$, denoted with the same symbol. 
Since \[\cal L_ZU^\flat = dU^\flat(Z,.)=0,\]
the flow sends level sets to level sets, i.e.,
$\phi_s(x)\in \cal U_{z(x)+s}$. Indeed, for each $x\in \cal U_{z(x)}$, the function $f(s):=z(\phi_s(x))\in \mathbb{R}$ satisfies
\[f'(s)=df_s(\del_s)=dz|_{\phi_s(x)} (Z)\equiv 1,\]
and hence $f(s)=z(\phi_s(x))=s+z(x)$.
Then we have a diffeomorphism
\begin{eqnarray*}\Psi: \cal (-\epsilon,\epsilon) \times \cal U_0&\longrightarrow& \{ y \in \mathcal{W} \mid |z(y)|< \epsilon \} \subset \mathcal{W},\\
(s,x)&\longmapsto& \phi_s(x),
\end{eqnarray*}
with  $\Psi (s,x)\in \cal U_{s}$.  Its inverse is given by
\[
\Psi^{-1}(x)=\big(z(x), \phi_{-z(x)}(x)\big)\in \cal I\times \cal U_0.
\]
It also satisfies
\[d\Psi_{(s,x)} (\del_s)=Z|_{\phi_s(x)}.
\]
Therefore, for the pulled back metric we have
\begin{eqnarray*}
\Psi^*\g(\del_s,\del_s)&=& \g_{\phi_s(x)}(Z,Z)\ =\ \frac{1}{\g(\text{grad}^{\g} z,\text{grad}^{\g} z)}
\ =\ \frac{1}{\g(U,U)},
\\
\Psi^*\g(\del_s,X)&=& \g_{\phi_s(x)}(Z,d\Psi_{(s,x)}(X))\ =\ 0,
\end{eqnarray*}
since $d\Psi_{(t,x)}(X)$ is  tangential to a level set whenever $X$ is,  and $\Z$
is a multiple of the gradient of $z$. Finally, ${\h}_s$ is given by
\[{\h}_s(X,Y)|_x:=\Psi^*\g(X,Y)=\g_{\Psi_s(x)} (d\phi_s|_x(X),d\phi_s|_x(Y)),\]
for $X,Y\in \cal U_0$. Hence,
$
\Psi^*\g = \mu^2ds^2+{\h}_s
$
with
\[{\mu}=\frac{1}{\sqrt{\g(\text{grad}^{\g} z,\text{grad}^{\g} z)}} = \frac{1}{u},\qquad\partial_s = \frac{1}{dz(\text{grad}^{\g} z)}\text{grad}^{\g} z = \frac{1}{u^2} U.\]
Setting $\mathcal{F} := \mathcal{U}_0 = z^{-1}(0)$, this gives the local form of the metric \eqref{nform-vf}.

Moreover, if $Z=\frac{1}{u^2}U$ is  complete, \cite[Proposition 8]{leistner-schliebner13} shows that 
the flow   of the lift of $\Z$ to the universal cover $\widetilde{\M}$ of $\M$ defines a {\em global} diffeomorphism $\Psi$ between  $\widetilde{\M}$  and $\rr \times \widetilde{\F}$, where $\widetilde {F} $ is the universal cover of a leaf $\cal F$ of $U^\perp$.

 Finally, we compute for $(\M,\g)$ as in formula \eqref{nform-vf} and $X,Y \in T\mathcal{F}$ the symmetric bilinear form $\W = -\frac{1}{u}\nabla U$ as follows:
\begin{eqnarray*}
\W(\partial_s, \partial_s) &= &\partial_s\left(\tfrac{1}{u}\right), \\
\W(\partial_s, X) &= &-u \cdot \g(\nabla_{\partial_s} \partial_s, X) = X \left( \tfrac{1}{u} \right), \\
\W(X,Y) &= &-u \cdot \g(\nabla_{X} \partial_s, Y) = -\tfrac{u}{2}\dot{\h}_s(X,Y).
\end{eqnarray*}
Clearly, this is equivalent to equations \eqref{www}. This finishes the proof of Theorem \ref{pf}.
$\hfill \Box$

\subsection{Complete Riemannian manifolds satisfying the constraint}  \label{complsec}
In order to obtain complete Riemannian manifolds satisfying the constraint, we will use the following lemma, which is a weaker version of forthcoming results in \cite{leistner-schliebner16}, see also \cite[Lemma 2]{CortesHanMohaupt12}.

\begin{Lemma}\label{lem2}
Let $\cal F$ be a  compact manifold with a $s$-dependent family of  Riemannian metrics  $\h_s$ and let  
$u$ be a bounded, positive smooth  function on $\M=\rr\times \cal F$. Then the metric
\[\g=\frac{1}{u^2}\, ds^2 + 
\h_s\]
on $\M$ is complete. 
\end{Lemma}
\begin{proof}
According to the decomposition $\M=\rr\times \cal F$ we can write every curve $\gamma:[a,b)\to \M$ as  $\gamma(t)=(s(t),x(t))$ with $s:[a,b)\to \rr$. Hence, 
\[ \g_{\gamma(t)}(\dot\gamma(t),\dot\gamma(t))= \left(\frac{ \dot s(t)}{u(\gamma(t))} \right)^2+ \h_{s(t)} (\dot x(t),\dot x(t)).\]
For a curve with  
$\g(\dot\gamma,\dot\gamma) \equiv c\in \R$  constant,   e.g. a geodesic,  $\h_s$ being positive definite shows that $0\le \h_s(\dot x,\dot x)\le c$ is bounded, 
and  $u$ bounded implies
\[
(\dot s)^2
=(u\circ \gamma)^2(c-\h_u(\dot x,\dot x)) \le c (u\circ \gamma)^2 \le c \sup u,
\]
showing that also  $\dot s:[a,b)\to \rr$ is bounded. Hence, if $b\in \rr$, the function  $s:[a,b)\to \rr$ is bounded and  its image lies in a compact set in $\rr$. Hence $s(b)=\lim_{t\to b}s(t)\in \R$ is well-defined.

Now assume that $(\M,\g)$ is incomplete. 
Let $\gamma:[a,b)\to  \M$ be a maximal geodesic with $b\in \rr$. Then $\gamma$ leaves every compact set in $\M$. Indeed, 
if $\gamma(t)$  remained in a compact set, then $\{\gamma(t_n)\}_{n \in \N}$ with $t_n \rightarrow b^-$ would have a convergent subsequence. However, $\{\gamma(t_n)\}$ is a Cauchy sequence for the metric $d_\g$ induced by the Riemannian metric $\g$. Hence $\{\gamma(t_n)\}$ converges, and thus $\gamma $ could be extended beyond $b$.
On the other hand, 
we have seen that the image of $s$ lies in a compact set in $\rr$. Hence, that $\gamma$ leaves every compact set in $\M=\rr\times \cal F$ is a contradiction to $\cal F$  compact.
\eprf

\section{Special Lorentzian holonomy and families of Riemannian metrics}
\label{holsec}
Based on the classification of indecomposable holonomy groups of Lorentzian manifolds with parallel null vector field \cite{bb-ike93,leistnerjdg}, we will now show how we can use Theorem \ref{cauchy-vf-theo} to construct Lorentzian manifolds with prescribed holonomy from families of Riemannian metrics.
Our aim in this section is to prove Theorem \ref{doublecone}.
\subsection{The screen bundle of $(\bM,\bg)$} \label{screensection}

To every Lorentzian manifold with parallel null vector field, in particular  to the data $(\bM,\bg,\;V)$ constructed via Theorem \ref{cauchy-vf-theo}, we can associate the \textit{screen bundle} \[\mathcal{S}:={\;V}^{\perp} / \;V \rightarrow \bM\] equipped with covariant derivative $\nabla_X^{\mathcal{S}} [Y] := \left[ \bnab_X Y \right]$. In contrast to the general case, however, our settting always yields a canonical realisation of $\mathcal{S}$ as a subbundle $S$ of $T\bM$ by means of the natural vector bundle isomorphism
\[
\begin{array}{rcl}
T\bM \supset S:=T^{\perp} \cap V^{\perp}\ni Y &\longmapsto&[Y]\in \mathcal{S}. \\
\searrow & &\swarrow\\
&\bM&\end{array}\]
This isomorphism  pulls back $\nabla^{\mathcal{S}}$ to the covariant derivative
\[\nabla^S := \mathrm{pr}_S \circ \bnab_{|S}, \]
which in turn is metric with respect to the positive definite screen metric $\bg^{S} := \bg_{S \times S}$. Having these identifications in mind, we also refer to $S$ as the screen. The screen construction is a useful tool when analyzing the holonomy of $(\bM,\bg)$, which by construction is contained in the stabiliser of a null vector, i.e. (up to conjugation), we have  \[\Hol(\bM,\bg) \subset \SO(n) \ltimes \mathbb{R}^n \subset \SO(1,n+1).\] 
(Note that the Lorentzian manifolds arising via Theprem \ref{cauchy-vf-theo} are time-orientable.) For any subgroup $\mathbf{G} \subset \SO(n) \ltimes \mathbb{R}^n$, let $\mathrm{pr}_{\SO(n)} \mathbf{G}$ denote its projection onto the $\SO(n)$-part. Then we have by construction
\begin{align}
\Hol(S,\nabla^{S}) \cong \mathrm{pr}_{\SO(n)} \Hol(\bM,\bg).
\end{align}

Recall that on $\bM$ the parallel null vector field $V$ decomposes into $V = \overline{u}T - N$. We next list useful formulas for the screen covariant derivative $\nabla^S$ and the screen curvature $\R^S$. By trivial extension, we will often view a section of $S \rightarrow \bM$ equivalently as element of $\mathfrak{X}(\bM)=\Gamma(T\bM)$ which is everywhere orthogonal to $T$ and $\;V$ and denote it with the same symbol. 

\begin{Lemma} \label{pra}
Let $\sigma \in \Gamma(S)$ and let $X,Y,Z \in \Gamma(T\bM)$. The following hold:
\begin{eqnarray*}
\nabla^S_Y \sigma &=&\bnab_Y \sigma - \frac{1}{\overline{u}} \bg(\sigma,\bnab_Y T)  \;V, \\
\R^S(X,Y)\sigma &=& \mathrm{pr}_S (\bR(X,Y)\sigma), \\
0&=&(\nabla^S_Z\R^S)(X,Y)+ (\nabla^S_X\R^S)(Y,Z) + (\nabla^S_Y\R^S)(Z,X).
\end{eqnarray*}
\end{Lemma}

\begin{proof}
These are straightforward calculations following directly from the various definitions, parallelity of $V$ as well as the symmetries and second Bianchi identity for $\bR$.
\end{proof}
For the data $(\bM,\bg,\;V)$ constructed via Theorem \ref{cauchy-vf-theo}, let 
\[S^{r,s} = \otimes^{r,s}S:= \underbrace{S^* \otimes...\otimes S^*}_{r \times } \otimes \underbrace{S \otimes...\otimes S}_{s \times} \rightarrow \bM \] denote the $(r,s)$-screen tensor bundle with covariant derivative induced by $\nabla^{S}$ and denoted with the same symbol. $\R^S$ also denotes the curvature operator of $(T^{r,s},\nabla^S)$.

Finally, we need to understand the pullback $S|_{\M} \rightarrow \M$ of the screen bundle $S \rightarrow \bM$ by means of the inclusion $\M = \{ 0 \} \times \M \hookrightarrow \bM$. By construction, it follows that $S|_{\M} = U^{\perp} \rightarrow \M$, whence the restriction $\sigma|_{\M}\in \Gamma(S|_{\M})$ of any $\sigma \in \Gamma(S)$ to $\M$ can be regarded as vector field on $\M$ which is orthogonal to $U$. On the vector bundle $U^\perp\to \M$ we have the connection that is induced by the Levi-Civita connnection of $\g$, $\nabla^{\perp} = \mathrm{pr}_{U^{\perp}} \circ \nabla^{\g}$. Then
\blem
For each  $X \in \mathfrak{X}(\M)$ and $\sigma \in \Gamma(S)$, we have
\begin{equation}
\begin{aligned} \label{hsf1}
\left(\nabla^S_X \sigma\right)|_{\M} = \nabla^{\perp}_X \sigma|_{\M}.  
\end{aligned}
\end{equation}
\elem
\bprf
It follows from formula \eqref{lcf} that 
\[
\nabla^S_X\sigma|_{\M}=\pr_S( \bnab_X\sigma|_{\M})=\pr_S( \nabla_X\sigma|_{\M})=\pr_{U^\perp}( \nabla_X\sigma|_{\M}).\]
\eprf
Now we describe the parallelity of  a section of $S^{r,s} \to\bM$ in terms of the corresponding section of the pulled back bundle $S|_\M$.

\begin{Proposition} \label{hlem}
Let $(\M,\g,U)$ be a Riemannian manifold satisfying the constraint \eqref{cons} and $(\bM,\bg)$ the Lorentzian manifold arising via Theorem \ref{cauchy-vf-theo}.  Then the $\nabla^S$-parallel sections of the bundle $S^{r,s} \rightarrow \bM$ are in one-to-one correspondence with  $\nabla^{\perp}$-parallel sections $\zeta$ of the pulled back bundle $S^{r,s}|_{\M} \rightarrow \M$, i.e., with
\begin{align} \nabla^{\perp}_{X} \zeta = 0, \quad\text{ for all $X \in T\M$.}\label{pse}
\end{align}
\end{Proposition}

\bprf
By the previous lemma it is clear that 
a parallel section of $T^{r,s}S \rightarrow \bM$ satisfies \eqref{pse}. 

On the other hand, let us assume condition \eqref{pse}.
We extend $\zeta$ to a section of $T^{r,s}S \rightarrow \bM$ by parallel transport in $V$-direction, i.e., such that  $\nabla_V^S \zeta = 0$. It then suffices to show that $\nabla^S_X \zeta = 0$ for $X \in T^{\bot}$. To this end we introduce the bundle \[H:=(T^\perp)^*\otimes T^{r,s}S \to \bM\] as well as the section $A \in \Gamma(H)$, given by
\beq A(X) := \nabla^S_X \zeta. 
\eeq 
Clearly, there are naturally induced covariant derivatives on $H$. For  $X \in T^{\bot}$ we compute, using the identities from Lemma \ref{pra} as well as $ \bnab V = 0$, that
\beq
(\nabla^S_V A)(X) &=& \nabla^S_V (A(X)) - A(\bnab_V X)
\\& = &\R^S(V,X)\zeta +\nabla^S_X\nabla^S_V\zeta+  \nabla^S_{[V,X]}\zeta - \nabla^S_{\bnab_VX}\zeta \\
&=& \R^S(V,X)\zeta  \\
&=&0.
\eeq
Thus, $\bnab^S_V A = 0$ which is a linear symmetric hyperbolic first order PDE for $A$. As $A|_{\M} = 0$ by assumption, we conclude $A \equiv 0$.                  
\eprf
Now we specify  Proposition \ref{hlem} for the situation when $(\M,g)$ is globally given as in Theorem~\ref{pf}, i.e., when $\M=\cal I\times \cal F$,  where $s\in \cal I$ with $\cal I$ and interval, $\rr$ or the circle,  $\cal F$ is a smooth manifold,  and $\g=u^{-2}ds^2+\h_s$ with a  smooth non-vanishing function $u$ on $\M$. Then, $U=u^2\partial_s$ and  we can express a
section $Z$ of the bundle $U^\perp \to \M$ by an $s$-dependent family of sections $Z_s$ of $T\cal F \to \cal F$. 
Differentiating such a section $Z_s$ in direction $X\in T\cal F$ we get the identity
\[\nabla^\perp_XZ =\nabla^{\h_s}_XZ, \]
where $\nabla^{\h_s}$ is the Levi-Civita connection of the metric $\h_s$. Differentiating in $s$-direction, by the Koszul formula,  we get for each $X\in T\cal F$ that 
\begin{equation}
\label{koszul}
2\g(\nabla^\perp_{\partial_s}Z_s,X)
=
\partial_s( \g(Z_s,X)) +\g( [\partial_s,Z_s],X)
=
(\cal L_{\partial_s}\g)(Z_s, X) +2 \g( [\partial_s,Z_s ],X),
\end{equation}
where $\cal L_{\partial_s}$ denotes the Lie derivative with respect to $\partial_s$ and where we assume that $[\partial_s,X]=0$. 
However, we have that $(\cal L_{\partial_s}\g)(X,Y)=(\cal L_{\partial_s}\h_s)(X,Y)$, whenever $X$ and $Y$ are tangential to $\cal F$. 
Hence, when dualising equation \eqref{koszul} with the metric $\h_s$ we get that 
\[ \nabla^\perp_{\partial_s}Z_s= \tfrac{1}{2}(\cal L_{\partial_s}\h_s)^\sharp(Z_s)+[\partial_s,Z_s ],\]
where $\sharp$ denotes the ($s$-dependent)  dualisation with respect to $\h_s$. 
%
%
Introducing the notation $\dot\h_s$ for $\cal L_{\partial_s}\h_s$ we can write this concisely as
\[ \nabla^\perp_{\partial_s}Z_s=[\partial_s,Z_s ]+ \tfrac{1}{2}\dot\h^\sharp_s(Z_s).\]
Using this, for a family of $1$-forms $\sigma_s\in \Gamma(T^*\cal F)$ we get 
\[ (\nabla^\perp_{\partial_s}\sigma_s)(X)=
\partial_s(\sigma_s(X))-\sigma_s(\nabla^\perp_{\partial_s}X)
=
(\cal L_{\partial_s}\sigma_s)(X)-
\tfrac{1}{2}\sigma_s\left((\dot \h_s)^\sharp(X)\right),\]
for all $X\in T\cal F$, i.e. that
\[\nabla^\perp_{\partial_s}\sigma_s = \dot\sigma_s+\tfrac{1}{2}\dot\h_s^\sharp \bullet \sigma_s,\]
where $\bullet $ denotes the natural action of endomorphisms on $1$-forms and $\dot\s_s:=\cal L_{\partial_s}\sigma_s$ is the Lie derivative.

This relation generalises to families of tensor fields $\sigma_s$ of higher rank and we obtain:
 
\begin{Corollary} \label{hcor}
Let $\cal F$ be a smooth manifold and $\h_s$ be a family of Riemannian metrics on $\cal F$, where $s\in \cal I$ with $\cal I$ being  an interval, $\rr$ or the circle, $u$ a non-zero smooth function on $\M=\cal I\times \cal F$ and $\g=u^{-2}ds^2+\h_s$ be the Riemannian manifold defined in \eqref{nform-vf}. Moreover, let  $(\bM,\bg)$ the Lorentzian manifold arising from $(\M,\g)$  via Theorem \ref{cauchy-vf-theo}.  Then there is a one-to-one correspondence between
\begin{enumerate}
\item sections $\overline{\sigma}$ of the bundle $S^{k,l}\to\bM$ such that $\nabla^S_X \overline\sigma =0$ for all $X\in T\bM$;
\item sections $\sigma$ of of the bundle $\otimes^{k,l}U^\perp\to \M$ such that $\nabla^\perp_Y \sigma=0$ for all $Y\in T\M$;
\item $s$-dependent families of sections $\sigma_s$ of $\otimes^{k,l}T\cal F \to \cal F$ with 
\begin{eqnarray}
\nabla^{\h_s}_Z\sigma_s&=&0,\quad\text{ for all }Z\in T\cal F,\label{fl1}\\
\dot \sigma_s &=&-\tfrac{1}{2}\dot\h_s^\sharp\bullet \sigma_s,\label{flow1}
\end{eqnarray}
where $\nabla^{\h_s}$ is the Levi-Civita connection of the metric $\h_s$, the dot denotes  the Lie derivative in $s$-direction, $\sharp$ the dualisation with respect to $\h_s$,  and $\bullet$ is the natural action of an endomorphism field on $\otimes^{k,l}T\cal F$., i.e., \begin{eqnarray*}
(\dot\h_s^\sharp\bullet \sigma_s)
(X_1,\ldots, X_k)
&=&
\dot\h_s^\sharp\left( \sigma_s
(X_1,\ldots, X_k)\right)
\\
&& - \sigma_s( \dot\h_s^\sharp(X_1),X_2,\ldots, X_k) - \ldots -  \sigma_s( X_1,\ldots, X_{k-1},\dot\h_s^\sharp(X_k)),
\end{eqnarray*}
for $X_1, \ldots , X_k\in T\cal F$.
\end{enumerate}
\end{Corollary}

This corollary will now provide us with a proof of Theorem~\ref{doublecone}.




\subsection{Lorentzian special holonomy and the proof of Theorem~\ref{doublecone}}
Here we use the result of Section \ref{screensection} to obtain a proof of  Theorem~\ref{doublecone}. 
In the setting of Theorem~\ref{doublecone} start with data $(\M,\g,\W,U)$  satisfying the initial condition \eqref{cons} and then 
first apply Theorem~\ref{pf} to conclude that these  data are given as  \[(\M= L \times \mathcal{F}, \g = \frac{1}{u^2}ds^2 + {\h}_s, U = u^2 \partial_s)\] solving \eqref{cons} for $\W$ as in \eqref{www}.  Thus the existence of $(\bM,\bg)$ with parallel null vector and initial data for $\g_t$ and $\dot{\g}_t$ as desired follows from Theorem \ref{cauchy-vf-theo}. Next, it follows from Section \ref{screensection}, in particular from Proposition \ref{hlem}, that $\mathrm{pr}_{\SO(n)} \Hol(\bM,\bg)=\Hol(S,\nabla^S)$ fixes an element in $T^{k,l} \mathbb{R}^n$ if and only if there is $\sigma \in \Gamma(\M,T^{k,l}U^{\perp})$ solving $\nabla^{\perp} \sigma = 0$. 
Using the explicit form of $(\M,\g)$ and $U$ from Theorem~\ref{pf}, $\sigma$ can be equivalently viewed as $s$-dependent family of tensor fields $\sigma_s \in \Gamma(\F,T^{k,l}\F) $. By and Corollary \ref{hcor}, equation $\nabla^{\perp} \sigma = 0$ is then equivalent to equations (\ref{fl1}) and (\ref{flow1}). This proves the first statement in Theorem \ref{doublecone} and it remains to verify the statements in Table \ref{table1}. For this we first consider the situation that the screen holonomy is in  $\mathbf{U}(\frac{n}{2})$, i.e., that \[\Hol(S,\nabla^S)=\mathrm{pr}_{\SO(n)}\Hol(\bM,\bg)\subset \mathbf{U}(\tfrac{n}{2}).\] 
By equation \eqref{fl1}, this case requires $\Hol(\mathcal{F},{\h}_s) \subset \mathbf{U}(\frac{n}{2})$. In other words, there are families of complex structures $J_s$, Kaehler forms $\omega_s$ on $\mathcal{F}$ which are parallel with respect to 
\begin{align}
{\h}_s = \omega_s(J_s \cdot, \cdot) \label{kaehlermetric}
\end{align} and satisfy 
the flow equations 
 \eqref{floweq}. Hence, 
 $\Hol(S,\nabla^S) \subset \mathbf{U}(\frac{n}{2})$ is equivalent to 
 $\Hol(\mathcal{F},{\h}_s) \subset \mathbf{U}(\frac{n}{2})$ and equations
 \begin{equation}
 \label{flowkaehler0}
\dot J_s +\tfrac{1}{2} \dot{\h}_s^{\sharp} \bullet J_s=0, \qquad
\dot \omega_s+ \tfrac{1}{2} \dot{\h}_s^{\sharp} \bullet \omega_s=0
\end{equation}
 for all $s$ and where the dot again denotes the Lie derivative with respect to the parameter $s$.

Now we turn to those holonomy groups in Table \ref{table1} that are defined  as the stabiliser of one or more tensors, i.e., to $\Sp(\frac{n}{4})$, $\G_2$ and $\Spin(7)$:

The case $n=4k$ and constraints for $\Hol(S,\nabla^S) \subset \mathbf{Sp}(k)$ is in 
 complete analogy to the $\mathbf{U}(\tfrac{n}{2})$-case, characterised by families of hyper-K\"ahler metrics ${\h}_s$ on $\F$ with corresponding compatible parallel almost complex structures $(J_s^1)$, $i=1,2,3$, i.e. $J_s^1 J_s^2 = J_s^3$ and Kaehler forms $\omega_s^i$ such that ${\h}_s = \omega_s(J_s \cdot, \cdot)$ satisfying the corresponding flow equations (\ref{flowkaehler0}).

For the case $n=7$ and constraints for $\Hol(S,\nabla^S) \subset \mathbf{G}_2$
recall that the exceptional group $\mathbf{G}_2 \subset \SO(7)$ can be realised as the stabiliser subgroup of a stable $3$-form in $\mathbb{R}^7$, see for example \cite{bryant87,joyce07,hitchin01,clss-09} more details. Hence, by Coroallary \ref{hcor} the case $\Hol(S,\nabla^S) \subset \mathbf{G}_2$ is characterised by 
a family of  associated stable 3-forms $\phi_s \in \Omega^3(\F)$ on $\mathcal{F}$ evolving according to equation~(\ref{flow1})
with associated family  
${\h}_s$ of $\mathbf{G}_2$. This implies 
 the corresponding entry in Table \ref{table1}.

For the case $n=8$ and constraints for $\Hol(S,\nabla^S) \subset \mathbf{Spin}(7)$ recall the 
 algebraic properties of the group $\mathbf{Spin}(7) \subset \SO(8)$ and its realisation in terms of the stabiliser of a generic 4-form, again see  \cite{bryant87,joyce07} for details but also  \cite{Karigiannis08}. The discussion then is completely analogous to the $\G_2$ case and the constraint equations are equivalent to the existence of a family of parallel $\mathbf{Spin}(7)$-structures $\psi_s$ on $\mathcal{F}$ evolving under the flow equation~(\ref{flow1}).

\bigskip

Now we turn to the case that 
is $n$ even and {\em the screen holonomy is special unitary}, i.e.,  $\Hol(S,\nabla^S) \subset \mathbf{SU}(\frac{n}{2})$. This is the most difficult case because this reduction is not simply given  as the stabiliser of a tensor, but rather by a trace condition in addition to the reduction to $\mathbf{U}(n)$. 

The parallel almost complex structures $J_s$ coming from the reduction $\Hol(S,\nabla^S) \subset \mathbf{SU}(\frac{n}{2})$ give a $\nabla^{\perp}$ parallel almost complex structure $J \in \Gamma(\M, \text{End}(U^{\perp}))$. By Proposition \ref{hlem},  $J$ gives via $\nabla^S$-parallel translation a $\nabla^S$-parallel almost complex structure $J^S$ on the screen $S \rightarrow \bM$. 
From now on we will work on the Lie algebra level. The holonomy algebra $\hol(S,\nabla^S)$ is contained in $\mathfrak{su}(\tfrac{n}{2})$ if and only if  $\hol(S,\nabla^S) \mathfrak{u}(\tfrac{n}{2})$ and each of its elements $A$ satisfies
\[\tr ( J^S\circ A)=0,\]
where we identify elements in the holonomy algebra with endomorphism of a fibre of the screen bundle $S$. 
Now we apply the  Ambrose-Singer holonomy theorem  to the holonomy algebra $\mathfrak{hol}(S,\nabla^S)$ at $p\in \M$, which states that 
\[\mathfrak{hol}(S,\nabla^S)=\mathrm{span}\{(P^S_{\gamma})^{-1}\circ \R^S(X,Y)\circ P^S_{\gamma}\mid \gamma:[0,1]\to \bM,\ \gamma(0)=p,X,Y\in T_{\gamma(1)}\bM\},
\]
where  $\R^S$ is the curvature of the screen bundle $S$ and $P^S_\gamma$ is the parallel transport along a curve $\gamma$. Since $J^S$ is parallel, it commutes with all parallel transports $P^S_\gamma$ and we obtain 
\[\tr( J^S \circ (P^S_{\gamma})^{-1}\circ \R^S(X,Y)\circ P^S_{\gamma} )= 
\tr( (P^S_{\gamma})^{-1}\circ J^S \circ  \R^S(X,Y)\circ P^S_{\gamma} )
=
\tr( J^S \circ  \R^S(X,Y) ).\]
Using this, the Ambrose-Singer theorem  and the fact that $\R^S(V,\cdot) = 0$ we obtain that 
 $\hol(S,\nabla^S) \subset \mathfrak{su}(\tfrac{n}{2})$ 
if and only if  $J^S$ additionally satisfies
\begin{align}
\mathrm{tr} \left( J^S \circ \R^S(X,Y) \right) = 0, \quad \text{ for all } \quad X,Y \in T^{\bot} \rightarrow \bM. \label{stre}
\end{align}
Let us now consider the left side of condition \eqref{stre} as section $C$ in the bundle $\Lambda^2 T^{\bot} \rightarrow \bM$, which in turn carries a covariant derivative induced by $\bnab$. We have by parallelity of $J^S$ and $V$ and Lemma \ref{pra} that
\beq
(\bnab_V C)(X,Y) &=&\mathrm{tr} \left( J^S \circ (\nabla_V^S R^S)(X,Y) \right) \\
&=&\mathrm{tr} \left( J^S \circ (\nabla_X^S R^S)(V,Y) \right) +\mathrm{tr} \left( J^S \circ (\nabla_Y^S R^S)(X,V) \right) \\
&=& 0.
\eeq
Hence, $C \equiv 0$ if and only if 
\begin{align}
C|_{\M} = 0, \label{eva1}
\end{align}
which in turn is evaluated by using the Gau\ss\ equation
 \be\label{gauss}
 \bR (X,Y,Z,L) \  = \ 
 \R (X,Y,Z,L)
 -\W(X,Z)\W(Y,L)+ \W(X,L)\W(Y,Z),
 \ee
for all $X,Y,Z,L\in T\M.$
Let $s_i$ be a local orthonormal basis of $S|_{\M} = U^{\perp} \rightarrow \M$ and $X,Y \in T\M$. Then we have
\begin{align*}
-\mathrm{tr} \left( J \circ \R^S(X,Y) \right)|_{\M} & = \sum_i \g(\R^S(X,Y) s_i, J^S(s_i) ) \\
& = \sum_i \bR(X,Y,s_i,J^S(s_i)) \\
& = \sum_i \R(X,Y,s_i,J(s_i)) - \W(X,a_i) \W(Y,J(s_i)) + \W(X,J(s_i)) \W(Y,s_i) \\
& = -\mathrm{tr} \left( J \circ \R(X,Y) \right) - \W(Y,J(\W(X))) + \W(X,J(\W(Y))).
\end{align*}
Therefore, the additional condition on $(\M,\g,\W,U,J)$ ensuring special unitary screen holonomy is
\begin{align}
\mathrm{tr} \left( J \circ \R(X,Y) \right) = - \W(Y,J(\W(X))) + \W(X,J(\W(Y))).\label{spun}
\end{align}
For Theorem \ref{doublecone} one has to evaluate equation  \eqref{spun} in terms of data on $(\F,{\h}_s)$ as in Theorem~\ref{pf}. Let $\W_s=\W_{|\F_s \times \F_s}$. Then one finds for the embedding $(\F,{\h}_s) \hookrightarrow (\M,\g)$ with unit normal $u\,\partial_s$ along $\mathcal{F}$ that
 \[ \nabla_X Y = \nabla^{{\h}_s}_X Y + \W_s(X,Y) u\cdot \partial_s, \text{ } \forall X,Y \in T\F. \]
 That is $\W_s$ is actually the Weingarten tensor of this embedding. Thus, using a Riemannian version of the Gau\ss\ equation, the curvature $\R_s$ of ${\h}_s$ is for $X,Y,Z,L \in T\F$ related to that of $(\M,\g)$ via
\begin{align}
\R(X,Y,Z,L) = \R_s(X,Y,Z,L) + \W_s(X,Z) \W_s(Y,L) - \W_s(X,L)\W_s(Y,Z) \label{riemgauss}
\end{align}
and the Codazzi equation
\begin{align}
\R(X,u\partial_s,Y,Z) = \left(d^{\nabla^{{\h}_s}}\W_s\right)(Y,Z,X) := (\nabla^{{\h}_s}_Y \W_s)(Z,X) - (\nabla^{{\h}_s}_Z \W_s) (Y,X). \label{codeq}
\end{align}
Inserting equation \eqref{riemgauss} into \eqref{spun}, we obtain for $X,Y \in T\F$ after a straightforward calculation using $\nabla^{{\h}_s} J_s = 0$, 
\begin{align*}
0 =\mathrm{tr}(J_s \circ \R_s(X,Y) ) =  - 2 \text{Ric}_s(X,J_s(Y)).
\end{align*}
On the other hand, we also need to evaluate
\begin{align} 
\mathrm{tr}\left( J \circ \R(X,\partial_s) \right) = - \W(\partial_s,J(\W(X))) + \W(X,J(\W(\partial_s))).\label{strg}
 \end{align}
The right side of \eqref{strg} is calculated using \eqref{www} and is equal to $\tfrac{2}{u^2}\g(\text{grad}^{\g} (u), J(\W(X)))$.
 For the left side, \eqref{spun} and \eqref{codeq} yield with a straightforward computation
\begin{align*}
 - (\delta^{\h_s}\dot{\h}_s)(J_s(X))  +\tfrac{2}{u^2} \g(\text{grad}^{\g} (u), J(\W(X)))
\end{align*}
Thus, \eqref{strg} is equivalent to
\begin{align}
(\delta^{\h_s}\dot{\h}_s) = 0 \label{divfree}
\end{align} 
Hence, the constraints for special unitary screen holonomy are equivalent to the existence of Ricci-flat K\"ahler metrics $(J_s,\omega_s,{\h}_s=\omega_s(J_s \cdot, \cdot))$ on $\F$ satisfying the flow equation (\ref{flow1}) and  additionally solve \eqref{divfree}.

\bigskip

Finally we consider  to the case when {\em the screen holonomy splits or is trivial}:
Suppose that there are proper subgroups $H_1$ and $H_2$ of $\SO(n)$ such that
\begin{align}
\Hol(S,\nabla^S) \subset {H}_1 \times {H}_2 \subset \SO(n) \label{spscr}
\end{align} 
Equivalently, there is a nontrivial, decomposable and $\nabla^S$-parallel form in the screen bundle. Thus, by Theorem \ref{floweq} and the holonomy principle, \eqref{spscr} is equivalent to a local metric splitting
\begin{align}
(\mathcal{F},{\h}_s) \cong (\mathcal{F}_1 \times \mathcal{F}_2, {\h}_s^1+{\h}_s^2)
\end{align} 
with $\Hol(\mathcal{F}_i,{\h}_s^i) \subset {H}_i$ and additionally the volume forms $\text{vol}^{{\h}_s^i}$ of the metrics ${\h}_s^i$, $i=1,2$ evolve according to
\begin{align}
\mathcal{L}_{\partial{s}} \text{vol}^{{\h}_s^i} =- \tfrac{1}{2} {{\dot{\h}_s^{i,{\sharp}}}} \bullet \text{vol}^{{\h}_s^i} \label{volel}.
\end{align}
However, it is well-known and straightforward to compute that \eqref{volel} holds for any time evolving metric with associated family of volume forms. 

Finally, let us now consider the special case that the screen is flat, i.e., the standard representation of $\Hol(S,\nabla^S)$ decomposes into $n$ trivial subrepresentations. It follows immediateley from an iterated version of the statement in the case where the screen holonomy splits hat this is equivalent to $(\mathcal{F},{\h}_s)$ being a family of flat metrics.
This proves Theorem \ref{doublecone}.
$\hfill \Box$
%
%
%


\subsection{One-parameter families of special Riemannian structures}\label{families}
Here we reformulate the evolutions equations (\ref{flow1}) in the case of K\"ahler and $\G_2$-structures further and formulate a question. We focus on one-parameter families of K\"ahler structures.
\begin{Lemma}\label{kaehlerlemma}
Let $(\cal F, h_s,J_s,\omega_s)$ be an $s$-dependent family of Riemannian K\"ahler structures on $\cal F$, i.e. with parallel complex structures $J_s$ and $\omega_s=h_s(J_s.,.)$ and set \[\Lambda^{1,1}(\cal F,J_s):=\{\phi\in \Lambda^2(\cal F)\mid \phi(J_sX,J_sY)=\phi(X,Y)\}.\] 
Then $\omega $ and $J$ satisfy the flow equations
\begin{equation}
 \label{flowkaehler}
\dot J_s +\tfrac{1}{2} \dot{\h}_s^{\sharp} \bullet J_s=0, \qquad
\dot \omega_s+ \tfrac{1}{2} \dot{\h}_s^{\sharp} \bullet \omega_s=0
\end{equation}
if and only if 
\begin{equation}
 \label{flowkaehler1}\dot \omega_s\in \Lambda^{1,1}(\cal F,J_s).\end{equation}
\end{Lemma}
\begin{proof} For brevity, we write drop the index $s$ indicating the $s$-dependence and write a dot for the Lie derivative with respect to $\partial_s$, i.e. $\dot{\omega}=\cal L_{\del_s}\omega$, etc.

First  compute 
\begin{equation}
\label{dothomega}
(\dot{\h}^{\sharp}\bullet \omega)(X,Y)=-\omega ( \dot{\h}^{\sharp}X,Y) -\omega(X, \dot{\h}^{\sharp} Y)
=
\dot h(X,JY)-
\dot h(JX,Y),
\end{equation}
which implies that  $(\dot{\h}^{\sharp}\bullet \omega)\in \Lambda^{1,1}(\cal F,J_s)$. This shows that equation (\ref{flowkaehler}) implies relation (\ref{flowkaehler1}).

Secondly,  Lie-differentiating and skew-symmetrising the relation 
$0=\omega-\h(J.,.)$ yields
\begin{eqnarray*}
\dot\omega(X,Y)&=&\tfrac{1}{2}\left( \dot \h(JX,Y)-\dot \h(X,JY) +\h(\dot JX,Y)-\h(X,\dot JY)\right)
\\
&=&
-\tfrac{1}{2}(\dot{\h}^{\sharp}\bullet \omega) +\tfrac{1}{2}\left(\h(\dot JX,Y)-\h(X,\dot JY)\right)
.
\end{eqnarray*}
by (\ref{dothomega}). 
We have seen that $(\dot{\h}^{\sharp}\bullet \omega)\in \Lambda^{1,1}(\cal F,J_s)$ and we claim that 
\begin{equation}\label{omega-claim}
\h(\dot JX,Y)-\h(X,\dot JY)\ \in\   \Lambda^{2}_-(\cal F,J_s):=\{\phi\in \Lambda^2(\cal F)\mid  \phi(J_sX,J_sY)=-\phi(X,Y)\},\end{equation}
which shows that relation (\ref{flowkaehler1}) implies equations (\ref{flowkaehler}). To prove claim (\ref{omega-claim})
we differentiate  $0=\omega(X,Y)-\omega(JX,JY)$ as in  \cite[Lemma 4.3]{StreetsTian14} and use 
 $\h(X,Y)=-\omega(JX,Y)=\omega(X,JY)$
to obtain that
\begin{eqnarray*}
0&=&
\dot\omega(X,Y)- \dot\omega(JX,JY)
-\omega(\dot JX,JY) -\omega(JX, \dot{J}Y)
\\
&=&
\dot\omega(X,Y)-\dot\omega(JX,JY)-
\h ( \dot{J}X,Y)+h(X, \dot{J} Y).
\end{eqnarray*}
This proves claim (\ref{omega-claim}) and because of $\Lambda^2(\cal F)=\Lambda^{1,1}(\cal F,J_s)\+ \Lambda^{2}_-(\cal F,J_s)$  establishes the desired equivalence.
\end{proof}
Note that not every family of K\"ahler stuctures $ (\h_s, J_s)$ satisfies equation (\ref{flowkaehler}):  for example for the constant family of flat metrics  $\h\equiv \h_s$ in even dimension the compatible complex structures are parametrised by the homogeneous space $\mathbf{GL}_n\C/\mathbf{U}(n)$. Taking a non constant curve of $\h$-parallel, i.e., constant,  complex structures $J_s$ gives a K\"ahler structure $(\h, J_s)$ with $\dot \h=0$ but $ \dot J_s\not=0$, which contradicts~(\ref{flowkaehler})\footnote{We would like to thank Vincente Cort\'{e}s for alerting us  to this example.}. 
Of course,  a constant family of constant complex structures $J_s\equiv J$ always satisfies  equation (\ref{flowkaehler}) for the flat metric $\h$. Clearly this suggest the following question:
 {\em given a family of K\"ahler metrics, is there a family of complex structures $J_s$, or of K\"ahler forms $\omega_s$, such that  
condition (\ref{flowkaehler1}), and hence flow equation (\ref{flowkaehler}) is satisfied?}

A difficulty when analysing equation 
(\ref{flowkaehler1}) arises from the fact that for an $s$-dependent family of complex structures $J_s$, the algebraic splitting of the two forms into $\Lambda^{1,1}$ and  $\Lambda^{2}_-$ {\em depends on the parameter $s$}.

For $\mathbf{G}_2$-structures the situation is similar.
Let $\phi_t$ be a family of $\G_2$-structures defining the family of Riemannian holonomy $\G_2$ metrics $\h_s$. Since the tangent space at a stable three form $\phi$ splits under $\G_2$ into three irreducible components
\begin{equation}\label{l3decomp}
\begin{array}{rcl}
\rr \+ Sym_0^2 (\rr^7)\+ \rr^7&\simeq &\Lambda^3\rr^7
\\[2mm]
(r,S,X)&\mapsto& r\phi+ S^\sharp\bullet \phi + X\hook (*\phi)
\end{array}
\end{equation}
it follows that   
\[\dot \phi = S^\sharp \bullet \phi +X\hook (*\phi),
 \]
 for a family of symmetric bilinear forms, whereas
   the associated metric satisfies
 \[\dot \h=2S,\]
see \cite{Bryant06g2,Karigiannis05,Karigiannis09}. Hence, similarly to the K\"ahler case, for the curve $\phi_t$ the 
equation
that results from Corollary \ref{hcor},
\begin{equation}\label{g2floweq}
\dot \phi +\tfrac{1}{2}\dot h^\sharp\bullet \phi=0\end{equation}
is equivalent to the condition 
\[\dot \phi\ \in\ \rr\oplus Sym_0(\rr^7),\]
i.e., that $\dot \phi$ has no $\rr^7$-component in the decomposition.
Again, it remains the question whether for a given family of parallel $\G_2$-structures $\h_s$ we can always find a corresponding family of stable $3$-forms $\phi_t$ satisfying this condition.
This suggest  to formulate the following

\begin{quest}
Let $(\bM,\bg)$ be a Lorentzian manifold  obtained from a Riemannian manifold $(\M,\g)$ satisfying the constraints via Theorem \ref{cauchy-vf-theo} and with screen holonomy $G=\mathrm{pr}_{\SO(n)}\Hol(\bM,\bg)$.
\begin{enumerate}
\item If $G\subset \mathbf{U}(\tfrac{n}{2})$, does there always exists a $\nabla^S$-parallel complex structure $J$ on $\cal S$ such that the associated family of $\h_s$-parallel and compatible complex structures  $J_s$  satisfies  the flow equation $\dot J_s+\tfrac{1}{2}\dot h^\sharp \bullet J_s=0$?
\item 
If $G\subset \mathbf{G}_2$ does there always exists a $\nabla^S$-parallel stable $3$-form $\phi$ on $\cal S$ such that the associated family of $\h_s$-parallel stable $3$-forms  $\phi_s$  satisfies  the flow equation $\dot \phi_s+\tfrac{1}{2}\dot h^\sharp \bullet \phi_s=0$?
\end{enumerate}
\end{quest}

\bbem
Interestingly, the flow equation \eqref{g2floweq} for the $\mathbf{G}_2$-case appears in \cite{Karigiannis09} in a completely unrelated context as $\mathbf{G}_2$-flow equation for not necessarily parallel 1-parameter families of $\mathbf{G}_2$-structures $\alpha_s \in \Omega^3(\mathcal{F})$ on $\F$. In fact, let $A_{ij}=A_{ij}(s)$ be any $s$-dependent family of symmetric $(2,0)$ tensors on $\F$ and consider the equation
\begin{align}
\partial_s \alpha_{ijk} = A_{i}^l \alpha_{ljk} + A^l_i \alpha_{ilk} + A_k^l \alpha_{ijl} \label{g2str}
\end{align}
for some given initial generic 3-form $\alpha_{s=0}$. As $\mathbf{G}_2 \subset \SO(7)$, every generic 3-form in dimension~$7$ yields a metrics ${\h}_s = {\h}_s(\alpha_s)$ in a natural way and \cite{Karigiannis09} then shows the relation
\begin{align*}
\partial_s h_{ij} = 2 A_{ij},
\end{align*}
which provides the link to our situation. However, it remains unclear under which conditions a parallel $\mathbf{G}_2$-structure $\alpha_0$ on $\F$ evolves under the flow equations \eqref{g2str} to a parallel family of $\mathbf{G}_2$-structures as required here.  In general, we have that (see \cite{Karigiannis09})
\[
\partial_s (\nabla^s_l \alpha_{ijk})  = A_i^m (\nabla^s_l \alpha_{mjk})  + A_j^m (\nabla^s_l \alpha_{imk}) + A_k^m (\nabla^s_l \alpha_{ijm}) + (\nabla^s A) \text{-terms}.
\]
To assure that a parallel $\mathbf{G}_2$-structure remains parallel under the flow, one would thus have to control the $\nabla^s A$-terms. This lies beyond the scope of this paper. The same discussion is possible on the level of $\mathbf{Spin}(7)$ structures and their flow equations which appeared in an unrelated context in \cite{Karigiannis08}.
\ebem

\section{Applications to Riemannian and Lorentzian spinor field equations}
 \label{igr}

Here we will use the previous results in order to obtain the two classification statements from Theorems \ref{imks} and \ref{nformsp} in the introduction.

\subsection{Generalised imaginary Killing spinors on Riemannian manifolds and the proof of Theorem \ref{imks}}
\label{ikssec}


Let us first suppose that $(\M,\g)$ admits an imaginary $\W$-Killing spinor. Differentiating \eqref{spinor-2}, it is easy to calculate that its Dirac current $U=U_{\varphi}$ defined by relation (\ref{diracU}) satisfies equation \eqref{genks}. Thus, Theorem \ref{cauchy-vf-theo} applies and there is a Lorentzian manifold $(\bM,\bg)$ in which $(\M,\g)$ embeds with second fundamental form $\W$ and $u\partial_t - U$ extends to a parallel null vector field $V$ on $\bM$.  
We extend the spinor $\varphi$ to a spinor $\phi$ on $\bM$ by parallel translation in direction of $V$, i.e. with  $\bnab_V \phi = 0$. Setting \[A(X) := \bnab_X \phi\] for $X \in \partial_t^{\bot}$, we find using parallelity of $V$ as well as the relations between spinorial and Riemannian curvature (for details see \cite{baum81}) that
\[
(\bnab_V A)(X) = \bnab_V \bnab_X \phi - \bnab_{\bnab_V X} \phi = \bR^{S^{\bg}}(V,X) \phi= \tfrac{1}{2} \bR(V,X) \cdot \phi =0.
\]
The well-known hypersurface formulas for the spinor covariant derivative \cite{baer-gauduchon-moroianu05} imply that the generalised Killing spinor equation for $\varphi$ is equivalent to $A|_{\M} = 0$ and as $A$ solves a linear first order symmetric hyperbolic PDE we conclude that $A \equiv 0$, and hence $\bnab \phi = 0$. In particular, $\Hol(\bM,\bg) \subset \SO(n) \ltimes \mathbb{R}^n$ fixes not only a parallel vector but also a nontrivial spinor. However, results in \cite{leistnerjdg,BaumLarzLeistner14} show that this can only happen if the screen holonomy satisfies that 
\begin{align}
\mathrm{pr}_{\SO(n)}(\Hol(\bM,\bg)) \subset {{H}}_1 \times....\times {{H}}_k, \label{amb}
\end{align}
with ${{H}}_i $ is equal to   $\mathbf{SU}(m_i)$, $\mathbf{Sp}(k_i)$, $\mathbf{G}_2$, $\mathbf{Spin}(7)$, or trivial.
By Theorem~\ref{pf}, $(\M,\g)$ is locally of the form $(\mathbb{R} \times \mathcal{F}, \frac{1}{u}^2 ds^2+ {\h}_s)$. Condition \eqref{amb} yields that locally $(\mathcal{F}_s,{\h}_s)$ splits into a metric product $(\mathcal{F}^1_s,h^1_s) \times...\times (\mathcal{F}^k_s,h^k_s)$ with $\Hol(\mathcal{F}^k_s,h^k_s) \subset {{H}}_k$. Moreover, Theorem \ref{doublecone} applied to this situation yields the evolution equations for ${\h}_s^i$ as given in Theorem~\ref{imks}. This proves (\ref{imks1}) in Theorem~\ref{imks}. The proof of the global version  in (\ref{imks2}) follows directly from the global statement in  Theorem~\ref{pf}.

Conversely, assume that $(\M,\g)$ is given as in the formulation of Theorem \ref{imks}. As an immediate consequence of Theorem \ref{doublecone}, $(\M,\g)$ embeds into a Lorentzian manifold $(\bM,\bg)$ with parallel null vector field $V$.  
Moreover we have that condition \eqref{amb} holds for the screen holonomy of $(\bM,\bg)$ which follows from Theorem \ref{doublecone}. However, in \cite{leistnerjdg,BaumLarzLeistner14} it is shown that each such Lorentzian holonomy group fixes a spinor whose Dirac current as defined in relation \eqref{diracc} is up to constant the null vector stabilised by $\SO(n) \ltimes \mathbb{R}^n$.
By the holonomy principle, there thus exists a parallel spinor $\phi$ with $V=V_{\phi}$. The well-known hypersurface formulas for the spinor covariant derivative in \cite{baer-gauduchon-moroianu05} imply that $\phi$ restricts to an imaginary $\W$-Killing spinor $\varphi = \phi|_{\M} \in\Gamma(\mathbb{S}^\g)$ on $\M$ with $\W$ being the Weingarten tensor of $\M\hookrightarrow \bM$  as given in Section \ref{prel}. It has been shown in \cite{BaumLeistnerLischewski16} that
\begin{align*}
U = \mathrm{pr}_{T\M} V|_{\M} = \mathrm{pr}_{T\M} {V_{\phi}}|_{\M} = U_{\varphi}.
\end{align*}
As $V_{\phi}$ is null, it follows that $V_{\phi} \cdot \phi = 0$, which evaluated on $\M$ gives precisely equation \eqref{spinor-2}. This shows (\ref{imks3}) in Theorem \ref{imks} and finishes the proof.
$\hfill \Box$

\subsection{Lorentzian holonomy and the proof of Theorem \ref{nformsp}}
For a manifold of the form \eqref{model} set $\mathcal{F} = \mathcal{F}_1 \times...\times \mathcal{F}_m$, $\h_w = \h_w^1 +...+\h_w^m$. Now introduce new coordinates by setting $v=-t+s$ and $w=t+s$, i.e. the metric in \eqref{model} becomes 
\begin{align}
\bg = -dt^2 + ds^2 + \h_{t+s} =: -dt^2 + \g_t \label{modelmet}
\end{align}
The metric \eqref{modelmet} admits a parallel null spinor if and only if $(\M:=\mathbb{R} \times \mathcal{F},\g_0)$ admits an imaginary generalised $\W$-Killing spinor additionally solving equation (\ref{spinor-2}). Indeed, it is clear that a parallel spinor restricts to an imaginary $\W$-Killing spinor on $(\M,\g_0)$. On the other hand, if a imaginary $\W$-Killing spinor $\varphi$ additionally solving equation (\ref{spinor-2}) exists on $(\M,\g_0)$ we use that $(\bM,\bg)$ admits a parallel null vector and exactly the same argument as in the proof of Proposition~\ref{hlem} or in Section~\ref{ikssec} extend $\varphi$ to a parallel null spinor for $(\bM,\bg)$. But with this equivalence, the statement follows immediately from the local classification result in Theorem \ref{imks}. 
$\hfill \Box$

\small

\providecommand{\MR}[1]{}\def\cprime{$'$} \def\cprime{$'$} \def\cprime{$'$}

\end{document}